\documentclass[a4paper]{article}

\usepackage[utf8]{inputenc}
\usepackage{amsmath}
\usepackage{amsfonts}
\usepackage{xypic}
\usepackage{amsthm}
\usepackage{mathrsfs}
\usepackage{amssymb}
\usepackage{amscd}
\usepackage[all]{xy}
\usepackage{accents}
\usepackage{float}
\usepackage{multirow}
\usepackage{color}
\usepackage[margin=3cm]{geometry}

\newlength{\dhatheight}
\newcommand{\doublehat}[1]{%
    \settoheight{\dhatheight}{\ensuremath{\hat{#1}}}%
    \addtolength{\dhatheight}{-0.35ex}%
    \hat{\vphantom{\rule{1pt}{\dhatheight}}%
    \smash{\hat{#1}}}}

\numberwithin{equation}{section}

\newtheorem{theorem}{Theorem}[section]
\newtheorem{lemma}[theorem]{Lemma}
\newtheorem{proposition}[theorem]{Proposition}
\newtheorem{corollary}[theorem]{Corollary}

\newtheorem{hypothesis}{Hypothesis}
\newtheorem{example}[theorem]{Example}

\theoremstyle{remark}
\newtheorem{remark}[theorem]{Remark}

\DeclareMathOperator{\Coh}{Coh}
\DeclareMathOperator{\QCoh}{QCoh}
\DeclareMathOperator{\Blow}{Blow}
\DeclareMathOperator{\K3}{K3}
\DeclareMathOperator{\Hilb}{Hilb}

\DeclareMathOperator{\Sym}{Sym}
\DeclareMathOperator{\sym}{\mathfrak{S}}
\DeclareMathOperator{\Pic}{Pic}
\DeclareMathOperator{\rk}{rk}

\DeclareMathOperator{\Ext}{Ext}

\DeclareMathOperator{\discr}{discr}

\DeclareMathOperator{\Vect}{Vect}

\DeclareMathOperator{\id}{id}

\DeclareMathOperator{\Ker}{Ker}
\DeclareMathOperator{\sign}{sign}

\DeclareMathOperator{\dif}{d}
\DeclareMathOperator{\tr}{tr}

\DeclareMathOperator{\et}{and}

\DeclareMathOperator{\ch}{ch}
\DeclareMathOperator{\GL}{GL}

\DeclareMathOperator{\Ree}{Re}
\DeclareMathOperator{\Ima}{Im}

\DeclareMathOperator{\pr}{pr}
\DeclareMathOperator{\Mon}{Mon}
\DeclareMathOperator{\divi}{div}
\DeclareMathOperator{\divides}{divides}
\DeclareMathOperator{\ou}{or}
\DeclareMathOperator{\Aut}{Aut}
\DeclareMathOperator{\Kum}{Kum}
\DeclareMathOperator{\Ort}{O}

\DeclareMathOperator{\Z}{\mathbb{Z}}
\DeclareMathOperator{\C}{\mathbb{C}}
\DeclareMathOperator{\Q}{\mathbb{Q}}

\newcommand{\Dd}{\mathcal{D}}

\newcommand{\Ii}{\mathcal{I}}
\newcommand{\Hh}{\mathcal{H}}

\newcommand{\Mm}{\mathcal{M}}

\newcommand{\Oo}{\mathcal{O}}
\newcommand{\Pp}{\mathcal{P}}

\newcommand{\Uu}{\mathcal{U}}

\newcommand{\CC}{\mathbb{C}}

\newcommand{\PP}{\mathbb{P}}

\newcommand{\RR}{\mathbb{R}}
\newcommand{\ZZ}{\mathbb{Z}}

\DeclareMathOperator{\M}{M}

\renewcommand{\L}{\mathbf{L}}

\newcommand{\R}{\mathbf{R}}

\newcommand{\Fff}{\mathscr{F}}
\newcommand{\Ggg}{\mathscr{G}}

\newcommand{\Ppp}{\mathscr{P}}

\newcommand{\BBB}{\mathrm{(BBB)}}
\newcommand{\BAA}{\mathrm{(BAA)}}
\newcommand{\ABA}{\mathrm{(ABA)}}
\newcommand{\AAB}{\mathrm{(AAB)}}

\newcommand{\quotient}[2]{{\raisebox{.2em}{\thinspace $#1$}\left / \raisebox{-.15em}{ $#2$}\right.}}

\newcommand\Quotient[2]{
        \mathchoice
            {
                \text{\raise1ex\hbox{\thinspace $#1$}\Big{/} \lower1ex\hbox{$#2$} \thinspace}%
            }
            {
                #1\,/\,#2
            }
            {
                #1\,/\,#2
            }
            {
                #1\,/\,#2
            }
    }

\newcommand\GIT[2]{
        \mathchoice
            {
                \text{\raise1ex\hbox{\thinspace $#1$}\Big{/}\!\!\!\!\Big{/} \lower1ex\hbox{$#2$} \thinspace}%
            }
            {
                #1\,/\,#2
            }
            {
                #1\,/\,#2
            }
            {
                #1\,/\,#2
            }
    }

\newcommand{\morph}[6]{\begin{array}{cccc} #6: & #1  & \stackrel{#5}{\longrightarrow} &  #2  \\ & #3 & \longmapsto & #4  \end{array}}

\newcommand{\ol}{\overline}

\date{\today}

\begin{document}

\title{Brane involutions on irreducible holomorphic symplectic manifolds}

\author{Emilio \textsc{Franco}\footnote{EF was supported by the FAPESP grants 2012/16356-6 and 2015/06696-2 (BEPE).}, Marcos \textsc{Jardim}\footnote{MJ was supported by FAPESP grant number 2014/05733-9 and 2016/03759-6, and also acknowledges the CNPQ grant number 303332/2014-0.}, Gr\'egoire \textsc{Menet}\footnote{GM is supported by the FAPESP grant number 2014/05733-9.}}

\maketitle



\begin{abstract}
In the context of irreducible holomorphic symplectic manifolds, we say that (anti)ho\-lo\-mor\-phic (anti)sym\-ple\-ctic involutions are brane involutions since their fixed point locus is a brane in the physicists' language, i.e. a submanifold which is either complex or lagrangian submanifold with respect to each of the three K\"ahler structures of the associated hyperk\"ahler structure. 

Starting from a brane involution on a $\K3$ or abelian surface, one can construct a natural brane involution on its moduli space of sheaves. We study these natural involutions and their relation with the Fourier--Mukai transform. Later, we recall the lattice-theoretical approach to Mirror Symmetry. We provide two ways of obtaining a brane involution on the mirror and we study the behaviour of the brane involutions under both mirror transformations, giving examples in the case of a $\K3$ surface and $\K3^{[2]}$-type manifolds. 
\end{abstract}

\maketitle

\centerline{\small{AMS Classification 2010: 14J28, 14J33, 14J50}}

\tableofcontents

\section{Introduction}

Branes play a central role in modern mathematical physics, specially in connection with the striking predictions of the mirror symmetry conjecture.

Following Kapustin and Witten \cite{kapustin&witten}, a {\it brane} in a hyperk\"ahler manifold is a submanifold which is either complex (B-brane) or lagrangian (A-brane) for each of the K\"ahler structures of our hyperk\"ahler structure. One then specifies the type of the brane, saying that it is either a $\BBB$, $\BAA$, $\ABA$ or an $\AAB$-brane. A hyperk\"ahler manifold is naturally a K\"ahler holomorphic symplectic manifold. In the compact case, by Yau's theorem, the converse is also true and both geometric structures are equivalent. In this context, a $\BBB$-brane is a holomorphic symplectic submanifold, and $\BAA$, $\ABA$ or $\AAB$-branes are complex lagrangian subvarieties for one of the three K\"ahler holomorphic symplectic structures. When non-empty, the fixed point locus of a holomorphic symplectic, a holomorphic antisymplectic, an antiholomorphic symplectic or an antiholomorphic antisymplectic involution is, respectively, a $\BBB$, $\BAA$, $\ABA$ or an $\AAB$-brane. Therefore, we refer to these involutions as {\it brane involutions}.

The study of holomorphic symplectic and antisymplectic involutions on $\K3$ surfaces was started by Nikulin \cite{Nikulin0, Lattice, Nikulin, NikulinIMU}. Since irreducible holomorphic symplectic (IHS) manifolds are the natural generalization of $\K3$ surfaces in higher dimension, many authors have extended this study to such manifolds \cite{Beauville, Sarti, Boissiere, BCS}. Aspinwall and Morrison \cite{Aspinwall} studied mirror symmetry for $\K3$ surfaces in terms of the cohomology lattice $H^2(X,\ZZ)$. Dolgachev in \cite{Dolgachev} constructed the moduli spaces of $M$-polarized K3 surfaces where $M$ is a sub-lattice of the Picard lattice; as a consequence, he provided a notion of mirror symmetry for the moduli spaces of K3 surfaces endowed with a holomorphic antisymplectic involution. This lattice-theoretical approach to mirror symmetry can be generalized to irreducible holomorphic symplectic manifolds. Huybrechts in \cite{Huybrechts} proposed a first definition via an action on the period domain and, Camere in \cite{Camere} extended Dolgachev construction to higher dimensions.

\

In this article we study brane involutions on irreducible holomorphic symplectic manifolds. From a brane involution $i$ on a symplectic surface $X$, it is possible to give a natural construction of a brane involution $\hat{\imath}$ on a moduli space $\M_X(H,v)$ of sheaves over $X$.  In Section \ref{sc moduli space} we prove that $\hat{\imath}$ has the same type as $i$ [Theorem \ref{tm branes in M}], and that this construction commutes with Fourier--Mukai transform [Proposition \ref{pr FM commutes with brane involutions}]. 

Under some conditions, we prove in Section \ref{discussion} the existence of mirrors, as defined by Huybrechts via an action on the period domain \cite{Huybrechts}. To be more precise, for a marked irreducible holomorphic symplectic manifold $(X,\varphi)$ endowed with a holomorphic 2-form $\sigma_X$ and a K\"ahler class $\omega_X$, we are able to associate a unique mirror set of data $(\check{X},\check{\varphi},\check{\sigma_X},\check{\omega_X})$. The proof is mainly a generalization of the transitivity of the Weyl group on the set of chambers of a K3 surface using the new developments of Mongardi \cite{NoteKalher} and Bayer--Hassett--Tschinkel \cite{Hassett} on the K\"ahler cone of an irreducible holomorphic symplectic manifold. Then it becomes possible to study the behaviour of brane involutions under mirror symmetry transformation  in Section \ref{sc branes and mirror symmetry}. From a brane involution $i$ on $X$ and the induced involution $i^*$ in cohomology, we obtain the {\it direct} mirror involution in cohomology $\check{\imath}_{dir}^{\, *}$ and the {\it indirect} mirror involution in cohomology $\check{\imath}_{ind}^{\, *}$ and we study under which conditions they promote to an involution $\check{\imath}$ in the mirror and how is the type transformed [Theorem \ref{tm mirror transform of branes}]. 
\[
\begin{tabular}{|c|c|c|}
\hline
Type of $i$ & Type of $\check{\imath}$ & Construction \\
\hline
$\BBB$ & $\BBB$ & Direct \\
\hline
\multirow{2}{1.2cm}{$\BAA$} & $\AAB$& Direct\\
\cline{2-3}
& $\BAA$ & Indirect\\
\hline
$\ABA$ & $\ABA$ & Direct\\
\hline
\multirow{2}{1.2cm}{$\AAB$} & $\BAA$ & Direct\\
\cline{2-3}
& $\AAB$ & Indirect\\
\hline
\end{tabular}
\]
In the case of holomorphic antisymplectic involutions, the indirect mirror transform corresponds to the one studied by Dolgachev \cite{Dolgachev} and Camere \cite{Camere}.

In Section \ref{sc examples} we provide concrete examples of such mirror transforms in the cases of $\K3$ surfaces and $\K3^{[2]}$-type manifolds endowed with a brane involution, studying the discrete lattice invariants that allow for such transformations. Finally, we give a description of $\ABA$ and $\AAB$-branes inside the Hilbert scheme of $2$ points of a $\K3$ surface.



\

\noindent {\it Acknowledgements.}

We want to thank Giovanni Mongardi and Richard Thomas for useful discussions. We also thank the Mathematical Institute of the University of Bonn, the Imperial College London, and the University of Edinburgh for their hospitality during the time this article was prepared.


\section{Preliminaries}

\subsection{Irreducible holomorphic symplectic manifolds}\label{sc IHSM}

An {\it irreducible holomorphic symplectic (IHS) manifold} $(X, \omega, \sigma)$ is a simply connected compact K\"ahler manifold $(X,\omega)$ such that $H^0(X, \Omega_X^2)$ is generated by an everywhere non-degenerate holomorphic 2-form $\sigma$. 

The current list of compact IHS manifolds, up to deformation equivalence, is:
\begin{itemize}
\item $\K3$ surfaces, the only examples in dimension 2;
\item Hilbert schemes of points of a $\K3$ surface;
\item generalized Kummer varieties; and
\item O'Grady's $6$- and $10$-dimensional examples.
\end{itemize}

Recall that a {\it hyperk\"ahler manifold} $(X,g, J_1, J_2, J_3)$ is a Riemannian manifold $(X,g)$ together with three complex structures $J_1$, $J_2$ and $J_3$ that satisfy the quaternion relations
$$ J_1^2 = J_2^2 = J_3^2 = J_1J_2J_3 = -1 $$ 
and such that the 2-forms $\omega_\ell(\cdot, \cdot) = g(\cdot, J_\ell(\cdot))$ are K\"ahler forms. Given a hyperk\"ahler manifold $(X,g,J_{1},J_{2},J_{3})$ one can always define a K\"ahler manifold equipped with a holomorphic symplectic structure $\sigma_1 = \omega_2 + i \omega_3$. For compact manifolds, by Yau's Theorem, the converse is also true, and giving a compact hyperk\"ahler manifold is equivalent to giving a compact holomorphic symplectic manifold.

For an IHS manifold $X$, one can endow $H^2(X, \ZZ)$ with the structure of a lattice $\Gamma$, where we denote by $\cdot$ the cup product in the case of $\K3$ surfaces, or the Beauville--Bogomolov form in higher dimensions. One has that $\Gamma = U^3 \oplus E_8(-1)^2$ in the case of a $\K3$ surface and $\Gamma = U^3 \oplus E_8(-1)^2 \oplus (2-2n)$ in the case of its Hilbert scheme of $n$ points.

Let $\Gamma$ be a lattice of signature $(3,r)$, with $r\in\mathbb{N}$, and let $X$ be an IHS manifold such that $H^{2}(X,\Z)\simeq \Gamma$. A \emph{mark} on $X$ is the choice of an isometry $\varphi: H^{2}(X,\Z)\rightarrow \Gamma$; the pair $(X,\varphi)$ is called a \emph{marked irreducible holomorphic symplectic manifold}. The quadruple $(X,\sigma_X,\omega_X,\varphi)$ consisting of an IHS manifold $X$, a 2-form $\sigma_X\in H^0(X,\Omega_X^2)$, a K\"ahler class $\omega_X$, and $\varphi$ is a mark on $X$ will be called a \emph{marked irreducible holomorphic symplectic manifold endowed with a hyperk\"ahler structure}. 

Let $\mathcal{M}_{\Gamma}$ denote the moduli space of marked IHS manifolds $(X,\varphi)$, and define the \emph{period domain} as follows:
$$\Omega:=\mathbb{P}\left\{ x\in \Gamma\otimes\C ~|~ x^{2}=0\ \et\ x\cdot \overline{x}>0\right\}.$$
The relation between $\mathcal{M}_{\Gamma}$ and $\Omega$ are strongly related is decribed in the following result due to Huybrechts.


\begin{theorem}[\cite{Huybrechts2}, Theorem 8.1] \label{period}
Let $\mathcal{M}_{\Gamma}^o$ be a non-empty connected component of $\mathcal{M}_{\Gamma}$.
The map so called \emph{period map}
\[
\morph{\Mm^o_\Gamma}{\Omega}{(X,\varphi)}{\varphi\left(H^{0}(\Omega_{X}^{2})\right),}{}{\Ppp}
\]
is surjective.
\end{theorem}

Let $X_1$ and $X_2$ be IHS manifolds. We say that the isomorphism $f: H^*(X_1,\Z) \stackrel{\cong}{\rightarrow} H^*(X_2,\Z)$ is a \emph{parallel-transport operator} if there exist a smooth and proper family $\pi : \mathcal{X}\rightarrow B$ of IHS manifolds over an analytic base $B$, points $b_i\in B$, isomorphisms $\psi_i : X_i \rightarrow \mathcal{X}_{b_i}$ for $i=1,2$, and a continuous path $\gamma: \left[ 0, 1 \right] \rightarrow B$, satisfying $\gamma(0)=b_1$, $\gamma(1)=b_2$, such that the parallel transport in the local system $R\pi_* \Z$ along $\gamma$ induces the homomorphism $\psi_{2*}\circ f\circ \psi_1^{*}: H^*(\mathcal{X}_{b_1},\Z)\stackrel{\cong}{\rightarrow} H^*(\mathcal{X}_{b_2},\Z)$. An isomorphism $g: H^{k}(X_1,\Z)\stackrel{\cong}{\rightarrow} H^{k}(X_2,\Z)$ is said to be a \emph{parallel-transport operator} if it is the $k$-th graded summand of a parallel-transport operator $f$ as above. 

One has the folowing Hodge theoretic Torelli theorem.

\begin{theorem}[\cite{Markman}, Theorem 1.3] \label{Torelli}
Let $X$ and $Y$ be IHS manifolds, which are deformation equivalent.
\begin{enumerate}
\item $X$ and $Y$ are bimeromorphic if and only if there exists a parallel transport operator $f : H^2(X,\ZZ) \to H^2(Y,\ZZ)$, which is an isomorphism of integral Hodge structures.

\item \label{it torelli} Let $f:H^{2}(X,\Z)\rightarrow H^{2}(Y,\Z)$ be a parallel transport operator which is an isomorphism of integral Hodge structures. There exists an isomorphism $\widetilde{f}:X\rightarrow Y$, such that $f=\widetilde{f}_*$, if and only if $f$ maps some K\"ahler class on $X$ to some K\"ahler class on $Y$.
\end{enumerate}
\end{theorem}

\subsection{Moduli spaces of sheaves and Fourier--Mukai transforms}

A natural way to obtain higher dimensional symplectic holomorphic manifolds is by con\-si\-de\-ring moduli spaces of stable sheaves on a compact algebraic symplectic surface. 

Suppose then that $X$ is a compact algebraic symplectic surface, i.e. a projective $\K3$ or abelian surface. Given a sheaf $F \to X$, its Hilbert polynomial, $P_F(x)$, is completely determined by a Mukai vector
\[
v := (r,D,s) \in \ZZ^{\geq 0} \oplus \Pic(X) \oplus \ZZ,
\]
where $r \in \ZZ^{\geq 0}$ gives the rank $\rk(F)$ of the sheaf, $D \in \Pic(X) \to H^2(X,\ZZ)$ 
determines its first Chern class $c_1(F)$, and $s$ is $\chi(F) - \varepsilon\rk(F)$, where $\varepsilon = 1$ if $X$ is a $\K3$ surface and $\varepsilon = 0$ when $X$ is an abelian surface. We denote by $P_v$ the Hilbert polynomial determined by the Mukai vector $v$. Fix a polarization $H$ and denote by $\M_X(H,v)$ to be the moduli space corre\-pre\-sen\-ting the moduli functor for the classification of $H$-stable sheaves with Hilbert polynomial equal to $P_v$.

In the case of rank $1$, $c_1 = 0$ and $c_2 = n$, the corresponding moduli space is known as the {\it Hilbert scheme} of $n$-points of $X$ and denoted by $\Hilb^n(X)$. When $X$ is an abelian surface with identity $x_0 \in X$, using the Hilbert-Chow map and the group law on $X$, one can produce a natural projection 
\begin{equation} \label{eq definition of Kummer}
a: \Hilb^n(X) \to X. 
\end{equation}
The kernel of this map, $ a^{-1}(x_0)$, is known as the {\it generalized Kummer variety} and denoted by $\Kum^n(X)$. 

\begin{theorem}[\cite{mukai1}, Corollary 0.2]
Given an abelian or $\K3$ surface $X$, then $\M_X(H,v)$ is a smooth quasiprojective scheme with a symplectic structure and each component has dimension $D^2 - 2rs + 2$. In particular, it has a trivial canonical bundle.
\end{theorem}
 
Recall that, for a $\K3$ surface, $\Pic(X)$ injects into $H^2(X,\ZZ)$. We say that $D \in \Pic(X)$ is {\it primitive} if it is indivisible as a cohomology class in $H^2(X,\ZZ)$. Let us denote by $\Hh$ the space of all possible polarizations on $X$, we know from \cite[Section 4.C]{huybrechts&lehn} that $v$ defines a locally finite set of real $1$-codimensional submanifolds called {\it $v$-walls}. We say that a {\it $v$-chamber} is each of the connected components of the complement of the union of all the $v$-walls. Note that the union of all $v$-chambers is dense in $\Hh$.

\begin{theorem}[\cite{huybrechts&lehn}, Theorem 6.2.5]  \label{tm M_X irreducible symplectic} 
Let $X$ be a $\K3$ surface and take a Mukai vector $v = (r, D, s) \in \ZZ \oplus \Pic(X) \oplus \ZZ$ with $D$ primitive. Suppose as well that the polarization $H$ on $X$ lies in the interior of a $v$-chamber. 

Then, the moduli space of stable sheaves $\M_X(H,v)$ is an IHS manifold.
\end{theorem}

Whenever $\M_X(H,v)$ is an IHS manifold, it is deformation equivalent either to a Hilbert scheme of points on a $\K3$ surface, or to a generalized Kummer variety.

We recall that the Zariski tangent space at the geometric point of $\M_X(H,v)$ given by the stable sheaf $F$ is 
\[
T_F \left ( \M_X(H,v) \right ) \cong \Ext^1(F, F).
\]
In \cite{mukai1}, the symplectic holomorphic $2$-form $\sigma_{\M}$ on the moduli space $\M_X(H,v)$ is given by the composition of the Yoneda product
\begin{equation} \label{eq yoneda product}
\Ext^{1}(F,F)\times \Ext^{1}(F,F) \longrightarrow \Ext^{2}(F,F),
\end{equation}
with the morphism
\[
\Ext^{2}(F,F)\stackrel{\tr}{\longrightarrow}H^{2}(X,\Oo_{X}) \stackrel{Serre}{\cong}
H^{0}(X,K_{X})^{\vee}\stackrel{\sigma_X}{\longrightarrow} \CC,
\]
defined by the trace, Serre duality and contraction by the holomorpic $2$-form $\sigma_X$ on $X$.


Given a $\K3$ surface $X$, thanks to the work of Mukai \cite{mukai2} one can ensure the existence of a universal family for the moduli problem of stable sheaves.

\begin{theorem}[\cite{mukai2}, Theorems A.5 and A.6; see also \cite{bartocci&bruzzo&ruiperez}, Theorem 4.20]  \label{tm existence of Uu} 
Let $X$ be a $\K3$ surface and $v = (r, D, s) \in \ZZ \oplus \Pic(X) \oplus \ZZ$ a Mukai vector such that the greatest common divisor of the integers $r$, $(D \cdot D')$ and $s$ (where $D'$ runs over all divisors in $\Pic(X)$) is $1$. 

Let $M$ be a connected component of the moduli space $\M_X(H,v)$ of $H$-stable sheaves over $X$ with Mukai vector $v$. Then there exists a universal family $\Uu \to X \times M$, and $M$ is a fine moduli space.
\end{theorem}

Denote by $p_X$ and $p_M$ the projections
\[
\xymatrix{
& X \times M \ar[ld]_{p_X} \ar[rd]^{p_M} &
\\
X & & M.
}
\]
Using the universal family $\Uu \to X \times M$ as a kernel, one can define Fourier--Mukai integral functors
\begin{equation} \label{eq definition of Fourier-Mukai F}
\morph{\Dd^b(X)}{\Dd^b(M)}{F^\bullet}{\R (p_M)_*\left(\Uu \stackrel{\L}{\otimes} p_X^*(F^\bullet) \right)}{}{\Fff}
\end{equation}
and
\begin{equation} \label{eq definition of Fourier-Mukai G}
\morph{\Dd^b(M)}{\Dd^b(X)}{E^\bullet}{\R (p_X)_*\left(\Uu \stackrel{\L}{\otimes} p_M^*(E^\bullet) \right).}{}{\Ggg}
\end{equation}

We say that a $\K3$ surface $X$ is {\it reflexive} if it carries a polarization $H$ and a divisor $D$ such that $H^2 = 2$, $H \cdot D = 0$, $D^2 = -12$, and $D + 2H$ is not effective. The interest on this class of $\K3$ surfaces follows from the next result.

\begin{proposition}[\cite{bartocci&bruzzo&ruiperez}, Proposition 4.35] 
\label{pr-definition-of-reflexive-K3}
If $X$ is a reflexive $\K3$ surface with polarization $H$, and Mukai vector $v = (2, D, -3)$, then there is a locally free rank $2$ universal sheaf $\Uu$ on $X \times \M_X(H,v)$, flat over 
$\M_X(H,v)$, making $\M_X(H,v)$ a fine moduli scheme parameterizing locally free stable sheaves with Mukai vector $v$. Moreover, $\hat{X} := \M_X(H,v)$ is a projective $\K3$ surface, and the Fourier--Mukai functor $\Fff : \Dd^b(X) \to \Dd^b(\hat{X})$ defined by $\Uu$ is a derived equivalence of categories.
\end{proposition}

We refer to the $\K3$ surfaces $X$ and $\hat{X}$ as {\it Fourier--Mukai partners}. Denote by $p_{\hat{X}}$ the projection $X \times \hat{X} \to \hat{X}$. Setting $\gamma$ to be the $(2,2)$-K\"unneth component of the Chern character $\ch(\Uu)$, one can construct \cite[(4.6)]{bartocci&bruzzo&ruiperez} the morphism
\[
\morph{\Pic(X)}{\Pic(\hat{X})}{D'}{(p_{\hat{X}})_*(D'\cdot \gamma).}{}{\mu}
\]
Accordingly, we denote
\[
\hat{H} := \mu(H)
\]
and, since $v = (2, D, -3)$,
\[
\hat{v} = (2, \mu(D), -3).
\]

\begin{theorem}[\cite{bartocci&bruzzo&ruiperez}, Theorem 4.47] \label{tm X cong hat hat X}
Let $X$, $H$ and $v$ be as in Proposition \ref{pr-definition-of-reflexive-K3}. The starting $\K3$ surface $X$ is isomorphic to $\doublehat{X} := \M_{\hat{X}}(\hat{H}, \hat{v})$. Furthermore, it is a fine moduli space with universal sheaf $\Uu^* \to \hat{X} \times X$.
\end{theorem}

The integral cohomology of Fourier--Mukai partners is related. 

\begin{theorem}[\cite{mukai2}, Theorem 1.5]  \label{tm FM induces a Hodge isometry}
Suppose that the projective $\K3$ surfaces $X$, $\hat{X}$ are Fourier--Mukai partners with universal sheaf $\Uu \to X \times \hat{X}$. Then one can construct the following map on integral cohomology
\[
\morph{\widetilde{H}(X,\ZZ)}{\widetilde{H}(\hat{X},\ZZ)}{\alpha}{(p_{\hat{X}})_*(p_X^*\alpha )\cdot v(\Uu),}{}{\psi}
\]
which is a Hodge isometry.
\end{theorem}

\subsection{Brane involutions}

Given a K\"ahler manifold $(X,J, \omega)$, we say that submanifold $X'\subset X$ is a A-{\it brane} if it is lagrangian with respect to $\omega$, and that it is B-{\it brane} if it is a complex submanifold with respect to $J$.

Now, let $(X,g,J_1, J_2, J_3)$ be a hyperk\"ahler manifold and denote by $\omega_1$, $\omega_2$ and $\omega_3$ the associated K\"ahler forms. A submanifold $Y\subset X$ is said to be a {\it brane} if it is either an A-brane (i.e. lagrangian) or a B-brane (complex) with respect to each of the $(J_\ell, \omega_\ell)$, for $\ell = 1$, $2$ or $3$. One then specifies the behaviour of $Y$ by saying that $Y$ is a brane of {\it type} $\BBB$, $\BAA$, $\ABA$ or $\AAB$; note that those are all the possible branes.

For instance, a $\BBB$-brane $Y\subset X$ is a hyperk\"ahler submanifold of $X$. In the other cases, if $Y\subset X$ is a $\BAA$, $\ABA$, or $\AAB$-brane if it is a complex-lagrangian submanifold with respect to the complex structure $J_1$, $J_2$ or $J_3$, respectively. 

As illustrated by \cite{baraglia&schaposnik} in the case of moduli spaces of Higgs bundles on curves, and by \cite{franco&jardim&marchesi} in the case of Nakajima quiver varieties, branes often arise as fixed point loci of involutions on a hyperk\"ahler manifold. Indeed, the following result is well known, see \cite{baraglia&schaposnik} or \cite[Section 2.2]{franco&jardim&marchesi} for instance. 

\begin{proposition} \label{pr branes as fixed points}
Let $(X,g,J_{1},J_{2},J_{3})$ be a hyperk\"ahler manifold, and let $i:X\to X$ be an analytic involution which preserves the metric $g$. If $i$ either commutes or anticommutes with each of the complex structures $J_\ell$, then the fixed point submanifold $X^{i}$ is a brane. 
\end{proposition}

If $i$ commutes with $J_\ell$, then $X^i$ is a complex submanifold of $(X,J_\ell)$; if $i$ anticommutes with $J_\ell$ then
$X^i$ is lagrangian with respect to $\omega_\ell$. This specifies whether $X^i$ is a $\BBB$, $\BAA$, $\ABA$ or a $\AAB$-brane.
We refer to an involution $i : X \to X$ satisfying the hypothesis of Proposition \ref{pr branes as fixed points} as a \emph{brane involution}. In particular, if the fixed point locus $X^i$ is a $\BBB$, $\BAA$, $\ABA$ or a $\AAB$-brane, we say that $i$ is an involution of \emph{type} $\BBB$, $\BAA$, $\ABA$ or $\AAB$, respectively.

Considering an IHS manifold $(X, J, \omega, \sigma)$, one can study the behaviour of these structures under a given isometric involution $i : X \to X$. Accordingly, we say that $i$ is
\begin{itemize}
\item \emph{holomorphic symplectic}, if it is holomorphic, and $i^{*}(\sigma)=\sigma$;

\item \emph{holomorphic antisymplectic}, if it is holomorphic, and $i^{*}(\sigma)=-\sigma$;

\item \emph{antiholomorphic symplectic}, if it is antiholomorphic, and $i^*(\sigma)=\overline{\sigma}$;

\item \emph{antiholomorphic antisymplectic}, if it is antiholomorphic, and $i^{*}(\sigma)=-\overline{\sigma}$.
\end{itemize}
Note, in addition, that $i^* \omega = \omega$ if $i$ is holomorphic, and $i^* \omega = - \omega$ if $i$ is antiholomorphic.

It is not difficult to check that the involutions above are brane involutions.

\begin{proposition} \label{pr biswas&wilkin}
Let $(X, g, J_1, J_2, J_3)$ be a hyperk\"ahler manifold with associated IHS manifold $(X, J_1, \omega_1, \sigma_1)$. Let $i : X \to X$ be an isometric involution; the fixed point locus $X^i$ is
\begin{itemize}
\item a $\BBB$-brane if $i$ is holomorphic symplectic;
\item a $\BAA$-brane if $i$ is holomorphic antisymplectic;
\item a $\ABA$-brane if $i$ is antiholomorphic symplectic; and
\item a $\AAB$-brane if $i$ is antiholomorphic antisymplectic.
\end{itemize}
\end{proposition}

\begin{proof}
The proof for the cases of holomorphic symplectic and holomorphic antisymplectic involutions are standard. The cases of antiholomorphic symplectic and antiholomorphic antisymplectic follow from \cite[Theorem 1.1]{biswas&wilkin}.
\end{proof}

According to Proposition \ref{pr biswas&wilkin}, we sometimes say that the type of the involution is the property of being (anti)holomorphic (anti)symplectic.

\begin{remark} \label{rm type in terms T and S}
For another point of view, denote by $i^* : H^2(X,\ZZ) \to H^2(X,\ZZ)$ the involution induced in cohomology. Set $T:=H^{2}(X,\Z)^{i^{*}}$ and $S:=T^\bot$, and note that $i^*$ acts trivially on $T$ while inverts $S$. It is straightforward to check that
\begin{itemize}
\item $\omega\in T \otimes \RR$, $\Ree \sigma \in T \otimes \RR$ and $\Ima \sigma \in T \otimes \RR$ for a $\BBB$-involution; 

\item  $\omega\in T \otimes \RR$, $\Ree \sigma \in S \otimes \RR$ and $\Ima \sigma \in S \otimes \RR$ for a $\BAA$-involution; 

\item $\omega\in S \otimes \RR$, $\Ree \sigma \in T \otimes \RR$ and $\Ima \sigma \in S \otimes \RR$ for an $\ABA$-involution; and

\item  $\omega\in S \otimes \RR$, $\Ree \sigma \in S \otimes \RR$ and $\Ima \sigma \in T \otimes \RR$ for an $\AAB$-involution.
\end{itemize}
\end{remark}

\subsection{Involutions on $\K3$ surfaces and Nikulin's description}\label{sc Nikulin}

After Nikulin's work \cite{Nikulin0, Lattice, Nikulin, NikulinIMU}, one has a precise description of holomorphic symplectic and antisymplectic involutions on a $\K3$ surface.

Given a holomorphic symplectic involution $i : X \to X$ on a $\K3$-surface, Nikulin showed \cite{Nikulin0} that it has a unique action on $H^2(X,\ZZ)$, up to isometry. He also described the fixed locus as $8$ isolated points (see also \cite{Sarti}),
\begin{equation} \label{eq fixed locus of a BBB-brane}
X^i = \bigsqcup_{k = 1}^8 \{ p_k \}.
\end{equation}

For a lattice $T$, we denote its rank by $r(T)$. The signature of $T$ will be denoted by $\sign(T)=(b^{+}(T),b^{-}(T))$.
A lattice $T$ is \textit{Lorentzian} if $\sign(T)=(1,r(T)-1)$. We denote by $T^{*}$ the dual of $T$ and by $A_{T}=T^{*}/ T$ the discriminant group. An even lattice $T$ is 2-\textit{elementary} if there is an integer $a$ with $A_{T}\simeq(\ZZ/2\ZZ)^{a}$; then we set $a(T)=\dim_{\ZZ/2\ZZ}(A_{T})$. We also define $\delta(T)=0$ if $x^{2}\in \ZZ$ for all $x\in T^{*}$, otherwise $\delta(T) = 1$. The triple $(r(T),a(T),\delta(T))$ determines the isometry class of an indefinite even Lorentzian 2-elementary lattice by \cite[Theorem 3.6.2]{Lattice}.

Nikulin \cite{Nikulin0, Nikulin} proved the following correspondence 
\begin{equation*}
\left\{
\begin{array}{rl}
&\textnormal{$\K3$ surface $X$ endowed}\\
&\textnormal{with a holomorphic}\\
&\textnormal{antisymplectic involution}
\end{array}\right\}\leftrightarrow 
\left\{
\begin{array}{rl}
&\textnormal{even Lorentzian 2-elementary}\\
&\textnormal{sublattices of $U^3\oplus E_{8}(-1)^2$}
\end{array}
\right\}
\end{equation*}
\begin{equation*}
(X,i)\ \ \ \ \ \ \ \ \ \ \ \ \ \ \ \ \ \leftrightarrow\ \ \ \ \ \ \ \ \ \ \  H^{2}(X,\Z)^{i^{*}}.
\end{equation*}

Then, to each holomorphic antisymplectic involution $i : X \to X$, we associate the values $(r,a,\delta)$ of the sublattice $H^{2}(X,\Z)^{i^{*}}$. Nikulin showed that there are 75 cases of 2-elementary sublattices of $U^3\oplus E_{8}(-1)^2$
with signature $(1,r)$, $r\leq 19$, represented in Figure \ref{fig BAA}. 

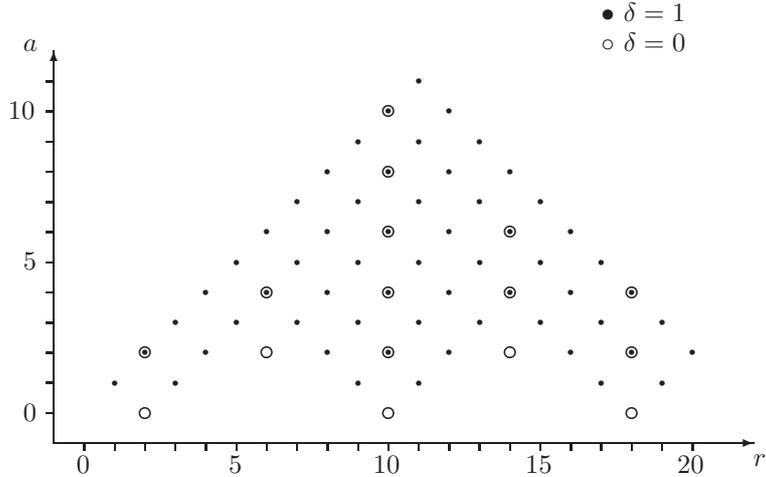
\begin{figure}[H]
\setlength{\unitlength}{2cm}
\centering
$$\begin{picture}(2,3)(1,0)
\put(-0.2,-0.2){\vector(1,0){4.6}}
\put(-0.2,-0.2){\vector(0,1){2.6}}

\put(0,-0.25){\line(0,1){0.05}}
\put(0.2,-0.25){\line(0,1){0.05}}
\put(0.4,-0.25){\line(0,1){0.05}}
\put(0.6,-0.25){\line(0,1){0.05}}
\put(0.8,-0.25){\line(0,1){0.05}}
\put(1,-0.25){\line(0,1){0.05}}
\put(1.2,-0.25){\line(0,1){0.05}}
\put(1.4,-0.25){\line(0,1){0.05}}
\put(1.6,-0.25){\line(0,1){0.05}}
\put(1.8,-0.25){\line(0,1){0.05}}
\put(2,-0.25){\line(0,1){0.05}}
\put(2.2,-0.25){\line(0,1){0.05}}
\put(2.4,-0.25){\line(0,1){0.05}}
\put(2.6,-0.25){\line(0,1){0.05}}
\put(2.8,-0.25){\line(0,1){0.05}}
\put(3,-0.25){\line(0,1){0.05}}
\put(3.2,-0.25){\line(0,1){0.05}}
\put(3.4,-0.25){\line(0,1){0.05}}
\put(3.6,-0.25){\line(0,1){0.05}}
\put(3.8,-0.25){\line(0,1){0.05}}
\put(4,-0.25){\line(0,1){0.05}}

\put(-0.05,-0.4){$0$}
\put(0.95,-0.4){$5$}
\put(1.90,-0.4){$10$}
\put(2.90,-0.4){$15$}
\put(3.90,-0.4){$20$}
\put(4.4,-0.35){$r$}

\put(-0.25,0){\line(1,0){0.05}}
\put(-0.25,0.2){\line(1,0){0.05}}
\put(-0.25,0.4){\line(1,0){0.05}}
\put(-0.25,0.6){\line(1,0){0.05}}
\put(-0.25,0.8){\line(1,0){0.05}}
\put(-0.25,1){\line(1,0){0.05}}
\put(-0.25,1.2){\line(1,0){0.05}}
\put(-0.25,1.4){\line(1,0){0.05}}
\put(-0.25,1.6){\line(1,0){0.05}}
\put(-0.25,1.8){\line(1,0){0.05}}
\put(-0.25,2){\line(1,0){0.05}}
\put(-0.25,2.2){\line(1,0){0.05}}

\put(-0.4,-0.05){$0$}
\put(-0.4,0.95){$5$}
\put(-0.5,1.95){$10$}
\put(-0.4,2.4){$a$}


\put(0.4,0){\circle{.07}}
\put(2,0){\circle{.07}}
\put(3.6,0){\circle{.07}}
\put(2,2){\circle{.07}}
\put(2,1.6){\circle{.07}}
\put(2,1.2){\circle{.07}}
\put(2,0.8){\circle{.07}}
\put(2,0.4){\circle{.07}}
\put(0.4,0.4){\circle{.07}}
\put(1.2,0.4){\circle{.07}}
\put(1.2,0.8){\circle{.07}}
\put(2.8,0.4){\circle{.07}}
\put(2.8,0.8){\circle{.07}}
\put(2.8,1.2){\circle{.07}}
\put(3.6,0.4){\circle{.07}}
\put(3.6,0.8){\circle{.07}}

\put(0.2,0.2){\circle*{.03}}
\put(0.6,0.2){\circle*{.03}}
\put(1.8,0.2){\circle*{.03}}
\put(2.2,0.2){\circle*{.03}}
\put(3.4,0.2){\circle*{.03}}
\put(3.8,0.2){\circle*{.03}}

\put(0.4,0.4){\circle*{.03}}
\put(0.8,0.4){\circle*{.03}}
\put(1.6,0.4){\circle*{.03}}
\put(2,0.4){\circle*{.03}}
\put(2.4,0.4){\circle*{.03}}
\put(3.2,0.4){\circle*{.03}}
\put(3.6,0.4){\circle*{.03}}
\put(4,0.4){\circle*{.03}}

\put(0.6,0.6){\circle*{.03}}
\put(1,0.6){\circle*{.03}}
\put(1.4,0.6){\circle*{.03}}
\put(1.8,0.6){\circle*{.03}}
\put(2.2,0.6){\circle*{.03}}
\put(2.6,0.6){\circle*{.03}}
\put(3,0.6){\circle*{.03}}
\put(3.4,0.6){\circle*{.03}}
\put(3.8,0.6){\circle*{.03}}

\put(0.8,0.8){\circle*{.03}}
\put(1.2,0.8){\circle*{.03}}
\put(1.6,0.8){\circle*{.03}}
\put(2,0.8){\circle*{.03}}
\put(2.4,0.8){\circle*{.03}}
\put(2.8,0.8){\circle*{.03}}
\put(3.2,0.8){\circle*{.03}}
\put(3.6,0.8){\circle*{.03}}

\put(1,1){\circle*{.03}}
\put(1.4,1){\circle*{.03}}
\put(1.8,1){\circle*{.03}}
\put(2.2,1){\circle*{.03}}
\put(2.6,1){\circle*{.03}}
\put(3,1){\circle*{.03}}
\put(3.4,1){\circle*{.03}}

\put(1.2,1.2){\circle*{.03}}
\put(1.6,1.2){\circle*{.03}}
\put(2,1.2){\circle*{.03}}
\put(2.4,1.2){\circle*{.03}}
\put(2.8,1.2){\circle*{.03}}
\put(3.2,1.2){\circle*{.03}}

\put(1.4,1.4){\circle*{.03}}
\put(1.8,1.4){\circle*{.03}}
\put(2.2,1.4){\circle*{.03}}
\put(2.6,1.4){\circle*{.03}}
\put(3,1.4){\circle*{.03}}

\put(1.6,1.6){\circle*{.03}}
\put(2,1.6){\circle*{.03}}
\put(2.4,1.6){\circle*{.03}}
\put(2.8,1.6){\circle*{.03}}

\put(1.8,1.8){\circle*{.03}}
\put(2.2,1.8){\circle*{.03}}
\put(2.6,1.8){\circle*{.03}}

\put(2,2){\circle*{.03}}
\put(2.4,2){\circle*{.03}}

\put(2.2,2.2){\circle*{.03}}

\put(0.4,0){\circle{.07}}
\put(2,0){\circle{.07}}
\put(3.6,0){\circle{.07}}
\put(2,2){\circle{.07}}
\put(2,1.6){\circle{.07}}
\put(2,1.2){\circle{.07}}
\put(2,0.8){\circle{.07}}
\put(2,0.4){\circle{.07}}
\put(0.4,0.4){\circle{.07}}
\put(1.2,0.4){\circle{.07}}
\put(1.2,0.8){\circle{.07}}
\put(2.8,0.4){\circle{.07}}
\put(2.8,0.8){\circle{.07}}
\put(2.8,1.2){\circle{.07}}
\put(3.6,0.4){\circle{.07}}
\put(3.6,0.8){\circle{.07}}

\put(3.4,2.6){$\bullet\ \delta=1$}
\put(3.4,2.4){$\circ\ \delta=0$}
\end{picture}$$\vspace{0.2cm}

\caption{$\BAA$-involutions}\label{fig BAA}
\end{figure}

The fixed locus of an involution $i$ can be described in terms of its invariants $(r,a,\delta)$.

\begin{theorem}[\cite{Nikulin0}] \label{tm nikulin description}
The fixed locus of a holomorphic antisymplectic involution
on a $\K3$ surface is
\begin{enumerate}
\item empty if $r = 10$, $a = 10$ and $\delta = 0$,
\item the disjoint union of two elliptic curves if $r = 10$, $a = 8$ and $\delta = 0$, 
\item the disjoint union of a curve of genus $g$ and $\ell$ rational curves otherwise, where $g = (22 - r - a)/2$ and $\ell = (r - a)/2$.
\end{enumerate}
\end{theorem}

\begin{remark} \label{rm other brane involutions in K3}
We can also use Nikulin's result to study anti\-ho\-lo\-morphic (anti)sym\-plec\-tic involutions $i : X \to X$ on a $\K3$ surface $X$. To do so, we consider a hyperk\"ahler rotation of our $\K3$ that sends $i$ to a holomorphic antisymplectic involution. It is then clear that $H^{2}(X,\Z)^{i^*}$ is also a 2-elementary Lorentzian sublattice of $U^3\oplus E_{8}(-1)^{2}$ and antiholomorphic (anti)symplectic involutions are also classified by the three integers $(r,a,\delta)$ of Figure \ref{fig BAA}. The fixed locus $X^i$ is the image of the hyperk\"ahler rotation of the fixed locus of a holomorphic antisymplectic involution described by Theorem \ref{tm nikulin description}. 
\end{remark}

\subsection{Involutions on $\K3^{[2]}$-type manifolds}

Holomorphic antisymplectic involutions on a $\K3^{[2]}$-type manifold can be classified by invariant lattices thanks to the work of Boissi\`ere, Camere and Sarti \cite{BCS}. 

In this case, $H^{2}(X,\ZZ) \cong \Gamma:=U^3\oplus E_8(-1)^2\oplus(-2)$, and the discriminant group $A_\Gamma=\Gamma^*/\Gamma$ is isomorphic to $\ZZ/2\ZZ$. Therefore, contrary to the $\K3$ case, for a $\K3^{[2]}$-type manifold $X$ the lattice $\Gamma$ is not unimodular, and the embedding of a Lorentzian 2-elementary sublattice $T$ into $\Gamma=U^3\oplus E_8(-1)^2\oplus(-2)$ is not unique: it depends on some parameters of $S=T^\bot$.

\begin{proposition}[\cite{BCS}, Proposition 8.2] \label{classificationTS}
Let $T$ be a Lorentzian 2-elementary lattice of signature $(1,t)$ and $\Gamma=U^3\oplus E_8(-1)^2\oplus(-2)$. Assume that $T$ admits a primitive embedding in $\Gamma$. We have that
\begin{itemize}
\item[(i)]
If there is no $x\in A_T$ such that $x^2=\frac{3}{2}\ \mod 2\Z$, then $T$ admits a unique primitive embedding into $\Gamma$ whose orthogonal complement is a 2-elementary lattice $S$ of signature $(2,20-t)$, $a(S)=a(T)+1$ and $\delta(S)=1$.
\item[(ii)]
Otherwise, non-isomorphic primitive embeddings of $T$ into $\Gamma$ are in 1-1 corres\-pon\-dence with non-isometric choices of 2-elementary lattice $S$ of signature $(2,20-t)$ with either $a(S)=a(T)-1$, or $a(S)=a(T)+1$ and $\delta(S)=1$.
\end{itemize}
\end{proposition} 

Such embeddings have a geometric realization.

\begin{proposition}[\cite{BCS}, Proposition 8.5]
For all embeddings $f:T\hookrightarrow \Gamma$ of $T$ a Lorentzian 2-elementary sublattice, there exists an IHS manifold of $\K3^{[2]}$-type with a non-symplectic involution $i:X\rightarrow X$ such that the invariant lattice $H^{2}(X,\Z)^{i}\subset H^2(X,\Z)$ is isomorphic to the embedding $f(T)\subset \Gamma$.
\end{proposition}

It follows that holomorphic antisymplectic involutions on a $\K3^{[2]}$-type manifold are classified by the values $(r(T),a(T),\delta(T))$ and $(r(S), a(S), \delta(S))$ occurring the following figures \cite[Figure 1 and 2]{BCS}.

\begin{figure}[H]
\setlength{\unitlength}{2cm}
\centering
$$\begin{picture}(2,3)(1,0)
\put(-0.2,-0.2){\vector(1,0){4.6}}
\put(-0.2,-0.2){\vector(0,1){2.6}}

\put(0,-0.25){\line(0,1){0.05}}
\put(0.2,-0.25){\line(0,1){0.05}}
\put(0.4,-0.25){\line(0,1){0.05}}
\put(0.6,-0.25){\line(0,1){0.05}}
\put(0.8,-0.25){\line(0,1){0.05}}
\put(1,-0.25){\line(0,1){0.05}}
\put(1.2,-0.25){\line(0,1){0.05}}
\put(1.4,-0.25){\line(0,1){0.05}}
\put(1.6,-0.25){\line(0,1){0.05}}
\put(1.8,-0.25){\line(0,1){0.05}}
\put(2,-0.25){\line(0,1){0.05}}
\put(2.2,-0.25){\line(0,1){0.05}}
\put(2.4,-0.25){\line(0,1){0.05}}
\put(2.6,-0.25){\line(0,1){0.05}}
\put(2.8,-0.25){\line(0,1){0.05}}
\put(3,-0.25){\line(0,1){0.05}}
\put(3.2,-0.25){\line(0,1){0.05}}
\put(3.4,-0.25){\line(0,1){0.05}}
\put(3.6,-0.25){\line(0,1){0.05}}
\put(3.8,-0.25){\line(0,1){0.05}}
\put(4,-0.25){\line(0,1){0.05}}

\put(-0.05,-0.4){$0$}
\put(0.95,-0.4){$5$}
\put(1.90,-0.4){$10$}
\put(2.90,-0.4){$15$}
\put(3.90,-0.4){$20$}
\put(4.4,-0.35){$r$}

\put(-0.25,0){\line(1,0){0.05}}
\put(-0.25,0.2){\line(1,0){0.05}}
\put(-0.25,0.4){\line(1,0){0.05}}
\put(-0.25,0.6){\line(1,0){0.05}}
\put(-0.25,0.8){\line(1,0){0.05}}
\put(-0.25,1){\line(1,0){0.05}}
\put(-0.25,1.2){\line(1,0){0.05}}
\put(-0.25,1.4){\line(1,0){0.05}}
\put(-0.25,1.6){\line(1,0){0.05}}
\put(-0.25,1.8){\line(1,0){0.05}}
\put(-0.25,2){\line(1,0){0.05}}
\put(-0.25,2.2){\line(1,0){0.05}}

\put(-0.4,-0.05){$0$}
\put(-0.4,0.95){$5$}
\put(-0.5,1.95){$10$}
\put(-0.4,2.4){$a$}


\put(0.4,0){\circle{0.07}}
\put(2,0){\circle{.07}}
\put(3.6,0){\circle{.07}}
\put(2,2){\circle{.07}}
\put(2,1.6){\circle{.07}}
\put(2,1.2){\circle{.07}}
\put(2,0.8){\circle{.07}}
\put(2,0.4){\circle{.07}}
\put(0.4,0.4){\circle{.07}}
\put(1.2,0.4){\circle{.07}}
\put(1.2,0.8){\circle{.07}}
\put(2.8,0.4){\circle{.07}}
\put(2.8,0.8){\circle{.07}}
\put(2.8,1.2){\circle{.07}}
\put(3.6,0.4){\circle{.07}}
\put(3.6,0.8){\circle{.07}}

\put(0.2,0.2){\circle*{.03}}
\put(0.6,0.2){\circle*{.03}}
\put(1.8,0.2){\circle*{.03}}
\put(2.2,0.2){\circle*{.03}}
\put(3.4,0.2){\circle*{.03}}
\put(3.8,0.2){\circle*{.03}}

\put(0.4,0.4){\circle*{.03}}
\put(0.8,0.4){\circle*{.03}}
\put(1.6,0.4){\circle*{.03}}
\put(2,0.4){\circle*{.03}}
\put(2.4,0.4){\circle*{.03}}
\put(3.2,0.4){\circle*{.03}}
\put(3.6,0.4){\circle*{.03}}
\put(4,0.4){\circle*{.03}}

\put(0.6,0.6){\circle*{.03}}
\put(1,0.6){\circle*{.03}}
\put(1.4,0.6){\circle*{.03}}
\put(1.8,0.6){\circle*{.03}}
\put(2.2,0.6){\circle*{.03}}
\put(2.6,0.6){\circle*{.03}}
\put(3,0.6){\circle*{.03}}
\put(3.4,0.6){\circle*{.03}}
\put(3.8,0.6){\circle*{.03}}

\put(0.8,0.8){\circle*{.03}}
\put(1.2,0.8){\circle*{.03}}
\put(1.6,0.8){\circle*{.03}}
\put(2,0.8){\circle*{.03}}
\put(2.4,0.8){\circle*{.03}}
\put(2.8,0.8){\circle*{.03}}
\put(3.2,0.8){\circle*{.03}}
\put(3.6,0.8){\circle*{.03}}

\put(1,1){\circle*{.03}}
\put(1.4,1){\circle*{.03}}
\put(1.8,1){\circle*{.03}}
\put(2.2,1){\circle*{.03}}
\put(2.6,1){\circle*{.03}}
\put(3,1){\circle*{.03}}
\put(3.4,1){\circle*{.03}}

\put(1.2,1.2){\circle*{.03}}
\put(1.6,1.2){\circle*{.03}}
\put(2,1.2){\circle*{.03}}
\put(2.4,1.2){\circle*{.03}}
\put(2.8,1.2){\circle*{.03}}
\put(3.2,1.2){\circle*{.03}}

\put(1.4,1.4){\circle*{.03}}
\put(1.8,1.4){\circle*{.03}}
\put(2.2,1.4){\circle*{.03}}
\put(2.6,1.4){\circle*{.03}}
\put(3,1.4){\circle*{.03}}

\put(1.6,1.6){\circle*{.03}}
\put(2,1.6){\circle*{.03}}
\put(2.4,1.6){\circle*{.03}}
\put(2.8,1.6){\circle*{.03}}

\put(1.8,1.8){\circle*{.03}}
\put(2.2,1.8){\circle*{.03}}
\put(2.6,1.8){\circle*{.03}}

\put(2,2){\circle*{.03}}
\put(2.4,2){\circle*{.03}}

\put(2.2,2.2){\circle*{.03}}

\put(3.4,2.6){$\bullet\ \delta(T)=\delta(S)=1$}
\put(3.4,2.4){$\circ \ \delta(T)=0,\ \delta(S)=1$}
\end{picture}$$\vspace{0.1cm}
\caption{$T\hookrightarrow \Gamma$ with $a(S)=a(T)+1$}\label{K32BAA1}
\end{figure}

\begin{figure}[H]
\setlength{\unitlength}{2cm}
\centering
$$\begin{picture}(2,3)(1,0)
\put(-0.4,-0.4){\vector(1,0){4.8}}
\put(-0.4,-0.4){\vector(0,1){2.8}}

\put(-0.2,-0.45){\line(0,1){0.05}}
\put(0,-0.45){\line(0,1){0.05}}
\put(0.2,-0.45){\line(0,1){0.05}}
\put(0.4,-0.45){\line(0,1){0.05}}
\put(0.6,-0.45){\line(0,1){0.05}}
\put(0.8,-0.45){\line(0,1){0.05}}
\put(1,-0.45){\line(0,1){0.05}}
\put(1.2,-0.45){\line(0,1){0.05}}
\put(1.4,-0.45){\line(0,1){0.05}}
\put(1.6,-0.45){\line(0,1){0.05}}
\put(1.8,-0.45){\line(0,1){0.05}}
\put(2,-0.45){\line(0,1){0.05}}
\put(2.2,-0.45){\line(0,1){0.05}}
\put(2.4,-0.45){\line(0,1){0.05}}
\put(2.6,-0.45){\line(0,1){0.05}}
\put(2.8,-0.45){\line(0,1){0.05}}
\put(3,-0.45){\line(0,1){0.05}}
\put(3.2,-0.45){\line(0,1){0.05}}
\put(3.4,-0.45){\line(0,1){0.05}}
\put(3.6,-0.45){\line(0,1){0.05}}
\put(3.8,-0.45){\line(0,1){0.05}}
\put(4,-0.45){\line(0,1){0.05}}

\put(-0.25,-0.6){$0$}
\put(0.75,-0.6){$5$}
\put(1.7,-0.6){$10$}
\put(2.7,-0.6){$15$}
\put(3.7,-0.6){$20$}

\put(4.4,-0.55){$r$}

\put(-0.45,-0.2){\line(1,0){0.05}}
\put(-0.45,0){\line(1,0){0.05}}
\put(-0.45,0.2){\line(1,0){0.05}}
\put(-0.45,0.4){\line(1,0){0.05}}
\put(-0.45,0.6){\line(1,0){0.05}}
\put(-0.45,0.8){\line(1,0){0.05}}
\put(-0.45,1){\line(1,0){0.05}}
\put(-0.45,1.2){\line(1,0){0.05}}
\put(-0.45,1.4){\line(1,0){0.05}}
\put(-0.45,1.6){\line(1,0){0.05}}
\put(-0.45,1.8){\line(1,0){0.05}}
\put(-0.45,2){\line(1,0){0.05}}
\put(-0.45,2.2){\line(1,0){0.05}}

\put(-0.6,-0.25){$0$}
\put(-0.6,0.75){$5$}
\put(-0.7,1.75){$10$}

\put(-0.6,2.4){$a$}


\put(0.4,0){\circle{0.07}}
\put(2,0){\circle{.07}}
\put(3.6,0){\circle{.07}} 
\put(2,2){\circle{.07}}
\put(2,1.6){\circle{.07}}
\put(2,1.2){\circle{.07}}
\put(2,0.8){\circle{.07}}
\put(2,0.4){\circle{.07}}
\put(0.4,0.4){\circle{.07}}
\put(1.2,0.4){\circle{.07}}
\put(1.2,0.8){\circle{.07}}
\put(2.8,0.4){\circle{.07}}
\put(2.8,0.8){\circle{.07}}
\put(2.8,1.2){\circle{.07}}
\put(3.6,0.4){\circle{.07}}
\put(3.6,0.8){\circle{.07}}

\put(0.2,0.2){\circle*{.03}}
\put(0.6,0.2){\circle*{.03}}
\put(1.8,0.2){\circle*{.03}}
\put(2.2,0.2){\circle*{.03}}
\put(3.4,0.2){\circle*{.03}}
\put(3.8,0.2){\circle*{.03}}

\put(0.4,0.4){\circle*{.03}}
\put(0.8,0.4){\circle*{.03}}
\put(1.6,0.4){\circle*{.03}}
\put(2,0.4){\circle*{.03}}
\put(2.4,0.4){\circle*{.03}}
\put(3.2,0.4){\circle*{.03}}
\put(3.6,0.4){\circle*{.03}}
\put(4,0.4){\circle*{.03}}

\put(0.6,0.6){\circle*{.03}}
\put(1,0.6){\circle*{.03}}
\put(1.4,0.6){\circle*{.03}}
\put(1.8,0.6){\circle*{.03}}
\put(2.2,0.6){\circle*{.03}}
\put(2.6,0.6){\circle*{.03}}
\put(3,0.6){\circle*{.03}}
\put(3.4,0.6){\circle*{.03}}
\put(3.8,0.6){\circle*{.03}}

\put(0.8,0.8){\circle*{.03}}
\put(1.2,0.8){\circle*{.03}}
\put(1.6,0.8){\circle*{.03}}
\put(2,0.8){\circle*{.03}}
\put(2.4,0.8){\circle*{.03}}
\put(2.8,0.8){\circle*{.03}}
\put(3.2,0.8){\circle*{.03}}
\put(3.6,0.8){\circle*{.03}}

\put(1,1){\circle*{.03}}
\put(1.4,1){\circle*{.03}}
\put(1.8,1){\circle*{.03}}
\put(2.2,1){\circle*{.03}}
\put(2.6,1){\circle*{.03}}
\put(3,1){\circle*{.03}}
\put(3.4,1){\circle*{.03}}

\put(1.2,1.2){\circle*{.03}}
\put(1.6,1.2){\circle*{.03}}
\put(2,1.2){\circle*{.03}}
\put(2.4,1.2){\circle*{.03}}
\put(2.8,1.2){\circle*{.03}}
\put(3.2,1.2){\circle*{.03}}

\put(1.4,1.4){\circle*{.03}}
\put(1.8,1.4){\circle*{.03}}
\put(2.2,1.4){\circle*{.03}}
\put(2.6,1.4){\circle*{.03}}
\put(3,1.4){\circle*{.03}}

\put(1.6,1.6){\circle*{.03}}
\put(2,1.6){\circle*{.03}}
\put(2.4,1.6){\circle*{.03}}
\put(2.8,1.6){\circle*{.03}}

\put(1.8,1.8){\circle*{.03}}
\put(2.2,1.8){\circle*{.03}}
\put(2.6,1.8){\circle*{.03}}

\put(2,2){\circle*{.03}}
\put(2.4,2){\circle*{.03}}

\put(2.2,2.2){\circle*{.03}}

\put(3.4,2.6){$\bullet\ \delta(T)=\delta(S)=1$}
\put(3.4,2.4){$\circ\ \delta(T)=1,\ \delta(S)=0$}
\end{picture}$$\vspace{0.5cm}
\caption{$T\hookrightarrow \Gamma$ with $a(S)=a(T)-1$}\label{K32BAA2}
\end{figure}

\begin{remark}
Let $X$ be a $\K3^{[2]}$-type IHS manifold endowed with a antiholomorphic (anti)symplectic involution. Using a hyperk\"ahler rotation we recover the case of a holomorphic antisymplectic involution. Then $H^{2}(X,\Z)^{i}$ is also a 2-elementary Lorentzian sublattice of $U^3\oplus E_{8}(-1)^{2}\oplus(-2)$ and the antiholomorphic (anti)symplectic involutions are also classified by the parameters of Figures \ref{K32BAA1} and \ref{K32BAA2}.
\end{remark}

If $i$ is a holomorphic involution on a K3 surface $X$, denote by $\hat{\imath}$ the natural involution on $\Hilb^2(X)$ given by the pull-back (see Corollary \ref{co natural involutions on Hilbert scheme} below),
\[
\morph{\Hilb^2(X)}{\Hilb^2(X)}{F}{i^*F.}{}{\hat{\imath}}
\]

Thanks to \cite[Section 4.2]{Boissiere} and the description of the fixed locus $X^i$ given in \eqref{eq fixed locus of a BBB-brane} and in Theorem \ref{tm nikulin description}, one can compute the fixed point locus of $\hat{\imath}$ obtaining the following well known descriptions. See also \cite{Mongardi} and \cite{BeauInvolution} for more general statements. 

\begin{corollary} 
Let $X$ be a smooth $\K3$ surface, and let $i:X\to X$ be a holomorphic symplectic involution on it. One has that 
\[
\Hilb^2(X)^{\hat{\imath}} \cong \,  \widetilde{X/i} \, \sqcup  \bigsqcup_{\ell = 1}^{\binom{8}{2} = 28}
\left\{ q_\ell \right\},
\]
where $\widetilde{X/i}$ is the blow-up of $X/i$ along the $8$ isolated fixed points of $X^i$ and it is a $\K3$ surface.

If $i$ is a holomorphic antisymplectic involution on $X$ associated to the lattice invariants $(r,a,\delta)$, then 
\begin{enumerate}
\item for $r = 10$, $a = 10$ and $\delta = 0$, one has
\[
\Hilb^2(X)^{\hat{\imath}} \cong \,  X/i, 
\]
\item for $r = 10$, $a = 8$ and $\delta = 0$, one has
\[
\Hilb^2(X)^{\hat{\imath}} \cong \,  (X/i) \, \sqcup \bigsqcup_{k = 1}^2 \Sym^2(E_k) \sqcup (E_{1} \times E_{2}),
\]
where $E_1$ and $E_2$ are elliptic curves, and

\item in the remaining cases,
\[
\Hilb^2(X)^{\hat{\imath}} \cong \,  (X/i) \, \sqcup \Sym^2(C) \sqcup \bigsqcup_{k = 1}^\ell \PP^2 \sqcup \bigsqcup_{k = 1}^\ell
(C \times \PP^1)  \sqcup \bigsqcup_{k = 1}^{\binom{\ell}{2}} (\PP^1 \times \PP^1). 
\]
where $C$ is a curve of genus $g$, where $g = (22 - r - a)/2$ and $\ell = (r - a)/2$.
\end{enumerate}
\end{corollary}


\section{Brane involutions on moduli spaces of sheaves}\label{sc moduli space}

\subsection{Algebraic antiholomorphic involutions}

Consider a complex manifold $X$ endowed with an antiholomorphic involution $i$. For a locally free sheaf $E$ on $X$, or equivalently a holomorphic vector bundle, the notion of real structure on $E$ is well known, and goes back to Atiyah \cite{atiyah}. We recall that {\it a real vector bundle} over $(X, i)$ is a holomorphic vector bundle $E \stackrel{\pi}{\to}X$ together with an involution $\tilde{\imath}:E \to E$, $\CC$-antilinear on the fibres and such that the diagram
\[
\xymatrix{
E \ar[r]^{\tilde{\imath}} \ar[d]^\pi & E \ar[d]^\pi
\\
X\ar[r]^i & X,
}
\]
commutes. With this construction, $i$ induces an involution on the moduli space of holomorphic vector bundles \cite{biswas&huisman&hurtubise, schaffhauser}, defined by $E \mapsto i^*\ol{E}$. Here $\ol{E}$ is the holomorphic vector bundle over $X$ whose underlying $C^\infty$-bundle is the same as $E$, but endowed with the conjugate complex structure, so that $i^*\ol{E}$ is also a holomorphic vector bundle over $X$. In \cite{Biswas}, the construction $F \mapsto i^* \ol{F}$ is generalized to analytic sheaves.

In order to further generalize this construction to algebraic sheaves, we need the following assumption. Consider the antiholomorphic involution of $\PP^n$
\[
\morph{\PP^n}{\PP^n}{[x_0 : \dots : x_n]}{[\ol{x}_0 : \dots : \ol{x}_n].}{}{\iota}
\]
Given an antiholomorphic involution on a quasiprojective variety $i : X \to X$, we say that $i$ is {\it algebraic} if there exists an embedding into a projective space, $X \hookrightarrow \PP^n$, such that $i$ coincides with the restriction of $\iota$.    

Recall that $i^*\ol{L}$ is well defined for any algebraic line bundle $L$ on a quasiprojective variety. Note that we can say, equivalently, that $i : X \to X$ is algebraic if there exists an ample line bundle $L$ such that $L \cong i^*\ol{L}$. In that case, the involution $\iota : \PP(H^0(X, L^{\otimes m})) \to \PP(H^0(X, L^{\otimes m}))$ is induced by the $\CC$-antilinear map $s \mapsto \overline{s \circ i}$.

We are interested in algebraic antiholomorphic involutions since they preserve the Zariski topology.

\begin{proposition}
Let $i : X \to X$ be an algebraic antiholomorphic involution on a quasiprojective variety. For any Zariski open subset $U \subset X$, $i(U)$ is a Zariski open subset of $X$. 
\end{proposition}

\begin{proof}
It is enough to prove the statement for the open subsets of the form $U_p = X \cap (\PP^n - \{ p = 0 \})$, where $p$ is the polynomial $p(x_0, \dots, x_n) = \sum a_{(j_0, \dots, j_n)} x_0^{j_0} \dots x_n^{j_n}$. Since the vanishing locus of a polynomial $p$ is the same as the vanishing locus of its complex conjugate $\ol{p} = \sum \ol{a}_{(j_0, \dots, j_n)} \ol{x}_0^{j_0} \dots \ol{x}_n^{j_n}$, we see that
\begin{align*}
i(U_p) = & \iota(X) \cap \iota  \left ( \PP^n -  \{ p = 0 \} \right ) 
\\
= & X \cap \left ( \PP^n - \iota \left ( \{ p = 0 \} \right ) \right ) 
\\
= & X \cap \left ( \PP^n - \{ p \circ \iota = 0 \} \right ) 
\\
= & X \cap \left ( \PP^n - \{ \ol{p \circ \iota} = 0 \} \right ),
\end{align*}
and therefore $i(U_p) = U_q$, where $q = \ol{p \circ \iota}$ is the polynomial $q(x_0, \dots, x_n) = \sum \overline{a}_{(j_0, \dots, j_n)} x_0^{j_0} \dots x_n^{j_n}$.
\end{proof}

One has naturally $i^*\ol{\Oo_X} \cong \Oo_X$, and observe that, for every $s \in \Oo_X(U)$, we have that $\ol{s \circ i} \in \Oo_X(i(U))$. Then for any sheaf of $\Oo_X$-modules $F$, we define the sheaf of $\Oo_X$-modules $i^*\ol{F}$ as the sheaf given by
\begin{equation} \label{eq definition of j*olFf 1}
i^*\overline{F}(U) := F(i(U)),
\end{equation}
with the $\Oo_X$-action defined as follows,
\begin{equation} \label{eq definition of j*olFf 2}
s \cdot f := (\ol{s \circ i}) f \in F(i(U)) = i^*\overline{F}(U),
\end{equation}
for any $s \in \Oo_X(U)$ and any $f \in \Oo_X(U)$. For every morphism of $\Oo_X$-modules $\varphi : F_1 \to F_2$, we define $i^*\ol{\varphi} : i^*\ol{F}_1 \to i^*\ol{F}_2$ to be the morphism of $\Oo_X$-modules given by
\[
i^*\ol{\varphi}|_U = \varphi|_{i(U)}. 
\]

\begin{remark} \label{rm autoequivalence of categories}
Since $i^*\ol{i^*\ol{F}}$ and $i^*\ol{i^*\ol{\varphi}}$ are naturally $F$ and $\varphi$, this construction gives an auto-equivalence of the category of sheaves of $\Oo_X$-modules.

Furthermore, due to the natural isomorphism $i^*\ol{\Oo_X} \cong \Oo_X$, we have that $i^* \ol{F}$ is (quasi)coherent if and only if $F$ is (quasi)coherent. Then $i^*\ol{(\bullet)}$ gives an auto-equivalence of $\Coh(X)$ and $\QCoh(X)$ whenever $i : X \to X$ is algebraic.  
\end{remark}

\begin{example}
If $X$ is the Fermat quartic in $\PP^3$, observe that the conjugation $\iota: \PP^3 \to \PP^3$ restricts to an antiholomorphic involution $i$ on $X$, which is algebraic by construction. 

Let $D$ be the divisor given by the intersection $X \cap \{ p = 0 \}$ and let $\Ii_D \subset \Oo_X$ be its ideal sheaf. Note that the affine subsets $U_i = \{ x_i \neq 0 \}$ are preserved by $\iota$ and observe that
\[
\Ii_D(U_i) \cong \quotient{\CC[x_0,x_1, x_2, x_3]}{\langle x_0^4 + x_1^4 + x_2^4 + x_3^4, x_i - 1, p\rangle.} 
\]
One can easily check that $i^*\ol{D} := i(D)$ is the intersection $X \cap \{ \ol{p \circ \iota} = 0 \}$, where $\ol{p \circ \iota}$ is a holomorphic polynomial, and  
\[
i^*\ol{\Ii_D}(U_i) \cong \quotient{\CC[x_0,x_1, x_2, x_3]}{\langle x_0^4 + x_1^4 + x_2^4 + x_3^4, x_i - 1,  \ol{p \circ \iota}\rangle.}
\]
\end{example}


\subsection{Natural involutions}

Now consider a brane involution $i : X \to X$ on a compact holomorphic symplectic surface $X$. If $i$ is holomorphic (i.e. of type $\BBB$ or $\BAA$), we set $\hat{\imath}$ to be the morphism of moduli spaces given by the pull-back,
\begin{equation} \label{eq definition of sigma_j holomorphic}
\morph{\M_X(H,v)}{\M_X(i^* H, i^* v)}{F}{i^*F.}{}{\hat{\imath}}
\end{equation}
When $H$ and $v$ are invariant under $i$, $\hat{\imath}$ defines an involution on the moduli space that we call {\it natural involution} associated to $i$. 

Given an algebraic antiholomorhic involution we can also define a natural involution on $\M_X(H,v)$. One can see that the Hilbert polynomial is preserved under $F \to i^*\overline{F}$. Note that any destabilizing subsheaf $G \subset i^*\overline{F}$ with respect to a linearization $i^*\overline{H}$ would give a destabilizing subsheaf $i^*\overline{G} \subset F$ with respect to the linearization $H$. Then, for an algebraic antiholomorphic brane involution $i$ (i.e. $\ABA$ or $\AAB$-type), we set
\begin{equation} \label{eq definition of sigma_j antiholomorphic}
\morph{\M_X(H,v)}{\M_X(i^*\overline{H},i^*\overline{v})}{F}{i^*\overline{F},}{}{\hat{\imath}}
\end{equation}
where, for every Mukai vector $v = (r, D, s) \in \ZZ \oplus \Pic(X) \oplus \ZZ$, we set $i^*\overline{v} = (r, i^*\overline{D}, s)$.

Note that $\hat{\imath}$ is an involution of $\M_X(H,v)$, whenever $H$ and $v$ are preserved by $i$, i.e. $i^*\overline{H} \cong H$ and $i^*\overline{v} \cong v$. In this case we say that $\hat{\imath}$ is a \emph{natural involution} induced by $i$. 


\begin{theorem} \label{tm branes in M}
Let $X$ be a projective $\K3$ or abelian surface, and let $i$ be an algebraic brane involution on $X$. Let $H$ and $D$ be a polarization and a divisor on $X$, respectively, both fixed by $i$. Let $\M_X(H, v)$ be the moduli space of $H$-stable sheaves on $X$ with Mukai vector $v = (r,D,s)\in \ZZ\oplus \Pic(X) \oplus \ZZ$. The involution
\[
\hat{\imath} : \M_X(H,v) \longrightarrow \M_X(H,v)
\]
is a brane involution of the same type as $i$.
\end{theorem}

\begin{proof}
The proof is straightforward in the holomorphic case, since $\hat{\imath}$ is simply the pull-back. However, we include this case in our proof for the sake of completeness. 

Take a $H$-stable sheaf $F$ with Mukai vector $v$ and, by abuse of notation, denote also by $F$ the geometric point of $\M_X(H,v)$ associated to it. Set $E = \hat{\imath}(F)$, and denote again by $E$ the associated sheaf. Note that $E \cong i^* F$ if $i$ is holomorphic, while $E\cong i^*\overline{F}$ if $i$ antiholomorphic.

Clearly, when $i$ is holomorphic, $\hat{\imath}$ is also holomorphic, so let us suppose that $i$ is antiholomorphic. The involution $\hat{\imath}$ induces 
\[
\morph{\Ext^k(F, F)}{\Ext^k(E, E)}{\psi}{i^*\overline{\psi}.}{}{\hat{\imath}^k}
\]
In the case $k=1$, this coincides with the differential of $\hat{\imath}$, 
\[
\hat{\imath}^1 = \dif \hat{\imath}.
\]
Recall that the complex structure $J_{\M}$ is given by multiplication by the imaginary number $\sqrt{-1}$; one gets
\[
\dif \hat{\imath}(\sqrt{-1} \psi) = i^*\overline{(\sqrt{-1} \psi)} = - \sqrt{-1} i^*\overline{\psi} = - \sqrt{-1} \dif \hat{\imath}(\psi).
\]
Therefore, $\hat{\imath}$ anticommutes with the complex structure $J_{\M}$,
\begin{equation} \label{eq behaviour of sigma_j on Gamma_n}
\dif \hat{\imath} \circ J_{\M} = - J_{\M} \circ \dif \hat{\imath}.
\end{equation}

Next, note that $\hat{\imath}^k$ commutes with the Yoneda product:
\begin{equation} \label{eq Yoneda commutes with sigma_j}
\xymatrix{
\Ext^1(F,F) \times \Ext^1(F,F) \ar[d]^{\hat{\imath}^1} \ar[r] & \Ext^2(F,F) \ar[d]^{\hat{\imath}^2}
\\
\Ext^1(E,E) \times \Ext^1(E,E) \ar[r] & \Ext^2(E,E).
}
\end{equation}
Recall that, for a $\K3$ or an abelian surface, $K_X \cong \Oo_X$, we can define $\hat{\imath}' : H^2(X,\Oo_X) \to  H^2(X,\Oo_X)$ and $\hat{\imath}'' : H^0(X,K_X) \to  H^0(X,K_X)$ to be given by $\hat{\imath}'(\varphi) = i^*\varphi$ and $\hat{\imath}''(s) = i^*s$ if $i$ is holomorphic, or $\hat{\imath}'(\varphi) = i^*\overline{\varphi}$ and $\hat{\imath}''_j(s) = i^*\overline{s}$, if $i$ is antiholomorphic. It follows that the diagram 
\begin{equation} \label{eq second part of the diagram}
\xymatrix{
\Ext^{2}(F,F) \ar[d]^{\hat{\imath}^2} \ar[r]^{\tr} \ar[d] & H^{2}(X,\Oo_{X}) \ar[r]^{Serre} \ar[d]^{\hat{\imath}'} & H^{0}(X,K_{X})^{\vee} \ar[d]^{\hat{\imath}''}
\\
\Ext^{2}(E,E) \ar[r]^{\tr} & H^{2}(X,\Oo_{X}) \ar[r]^{Serre} & H^{0}(X,K_{X})^{\vee},
}
\end{equation} 
commutes.

Since $i$ is a brane involution,
\[
i^* \sigma_X = b \cdot \sigma_X \quad \textnormal{or} \quad i^* \sigma_X = b \cdot \overline{\sigma}_X,
\]
depending on the type of $i$. As a consequence, the following diagram commutes,
\begin{equation} \label{eq third part of the diagram for Omega_M}
\xymatrix{
H^{0}(X,K_{X})^{\vee} \ar[r]^{\quad \quad \sigma_X} \ar[d]^{\hat{\imath}''} & \CC \ar[d]^{b \cdot g}
\\
H^{0}(X,K_{X})^{\vee} \ar[r]^{\quad \quad \sigma_X} & \CC,
}
\end{equation} 
where $g$ is the identity if $i$ is holomorphic, and the conjugation if it is antiholomorphic.

Observe now that the composition of the top rows of \eqref{eq Yoneda commutes with sigma_j}, \eqref{eq second part of the diagram} and \eqref{eq third part of the diagram for Omega_M} define $\sigma_{\M}$, while the bottom rows give $\hat{\imath}^* \sigma_{\M}$. Then, we find that
$$ \hat{\imath}^*\sigma_{\M} = b \cdot \sigma_{\M} \quad \textnormal{or} \quad \hat{\imath}^*\sigma_{\M} = b \cdot \overline{\sigma}_{\M}, $$
depending on whether $i$ is holomorphic or antiholomorphic. This proves that $\hat{\imath}$ has the type of $i$.
\end{proof}

Since Hilbert schemes are particular examples of moduli spaces of sheaves, we obtain the following corollary.

\begin{corollary} \label{co natural involutions on Hilbert scheme}
Let $X$ be a $\K3$ or abelian surface and let $i$ be an algebraic brane involution on $X$ inducing the natural involution $\hat{\imath}$ on $\Hilb^n(X)$. Then the fixed point locus $\Hilb^n(X)^{\hat{\imath}}$ is a brane within $\Hilb^n(X)$ of the same type as $i$.
\end{corollary}

When $X$ is an abelian surface, it is straightforward to check that the projection in equation \eqref{eq definition of Kummer} commutes with natural involutions, so we have the commutative diagram 
\[ \xymatrix{
\Hilb^n(X) \ar[d]^a \ar[r]^{\hat{\imath}} & \Hilb^n(X) \ar[d]^a
\\
X \ar[r]^{i} & X. } \]
If the involution on $X$ preserves the origin, the natural involution preserves the associated generalized Kummer varie\-ty.  

\begin{corollary} \label{co natural involutions on Kummer}
Let $X$ be an abelian surface and let $i$ be an algebraic brane involution on $X$ preserving the origin $x_0 \in X$. This induces a natural involution $\hat{\imath}$ on the generalized Kummer variety $\Kum^n(X)$ whose fixed point locus $\Kum^n(X)^{\hat{\imath}}$ is a brane inside $\Kum^n(X)$ of the same type as $i$.
\end{corollary}


\subsection{The Fourier--Mukai transform}

In this section, $X$ is a projective $\K3$ surface. 

Note that, in principle, the involution $\hat{\imath} : \M_X(H,v) \to \M_X(H,v)$ defined in \eqref{eq definition of sigma_j holomorphic} and \eqref{eq definition of sigma_j antiholomorphic} might not preserve a given connected component $M$ of $\M_X(H,v)$. However, under the hypothesis of Theorem \ref{tm M_X irreducible symplectic}, $\M_X(H,v)$ is connected and this issue does not arise. In that case, we simplify the notation by denoting 
\[
M := \M_X(H,v).
\]
Suppose as well that the hypotheses of Theorem \ref{tm existence of Uu} are satisfied, so that there exists a universal family $\Uu \to X \times M$. We study the behavior of $\Uu$ under the pull-back by the involution $(i \times \hat{\imath})$. 

\begin{lemma} \label{lm Uu commutes with delta}
Let $X$ be a projective $\K3$ surface with a holomorphic (resp. algebraic antiholomorphic) involution $i$ on it. Suppose that $v$ and $H$ are preserved by $i$ and satisfy the conditions of Theorems \ref{tm existence of Uu} and \ref{tm M_X irreducible symplectic}. Then, there exists a canonical isomorphism
\[
\Uu \cong (i \times \hat{\imath})^* \Uu \quad \left(resp. \thinspace \cong (i \times \hat{\imath})^* \overline{\Uu} \right).
\]
\end{lemma}

\begin{proof}
We can assume that $i$ is antiholomorphic since the proof for the holomorphic case is analogous. Due to Theorem \ref{tm branes in M}, we have that $(i \times \hat{\imath})$ is antiholomorphic as well. Then, $(i \times \hat{\imath})^* \overline{\Uu}$ is a holomorphic family of $H$-stable sheaves with Mukai vector $v$ parametrized by $M$. 

Let us denote $\Uu_y := \Uu |_{X \times \{ y \}}$ for every point $y \in M$. By definition of $\hat{\imath}$, one has that
\[
\Uu_{\hat{\imath} (y)} \cong i^*(\overline{\Uu_{y}}).
\]
This implies that $(i \times \hat{\imath})^* \overline{\Uu}$ is also a family classifying all the $H$-stable sheaves with Mukai vector $v$. The proof follows from the universality of $\Uu$ stated in Theorem \ref{tm existence of Uu}.
\end{proof}

This allows us to study the relation between the natural involution $\hat{\imath}$ and Fourier--Mukai functors. Suppose that $i : X \to X$ is a holomorphic involution. We denote by $\delta_i$ the pull-back by $i$ on $\Dd^b(X)$, 
\begin{equation} \label{eq definition of delta_j holomorphic}
\morph{\Dd^b(X)}{\Dd^b(X)}{F^\bullet}{i^*(F^\bullet).}{}{\delta_i}
\end{equation}
Thanks to Remark \ref{rm autoequivalence of categories}, if $i$ is an algebraic antiholomorphic involution, we set
\begin{equation} \label{eq definition of delta_j antiholomorphic}
\morph{\Dd^b(X)}{\Dd^b(X)}{F^\bullet}{i^*(\overline{F^\bullet}).}{}{\delta_i}
\end{equation}
Note that, in both cases, $\delta_i$ is an involution.

\begin{proposition} \label{pr FM commutes with brane involutions}
Suppose $X$ is a projective $\K3$ surface and let $i : X \to X$ be an algebraic brane involution. Let $H$ be a polarization on $X$ fixed by $i$, and let $v$ be a Mukai vector fixed by $i$ and satisfying the hypothesis of Theorems \ref{tm existence of Uu} and \ref{tm M_X irreducible symplectic}.

Consider the moduli space $M := \M_{X}(H,v)$ and recall the involution $\hat{\imath} : M \to M$ defined in equations \eqref{eq definition of sigma_j holomorphic} and \eqref{eq definition of sigma_j antiholomorphic}. For every $F^\bullet \in \Dd^b(X)$ and every $E^\bullet \in \Dd^b(M)$, one has
\begin{equation} \label{eq Fff commutes with delta}
\Fff(\delta_i(F^\bullet)) \cong \delta_{\hat{\imath}}(\Fff(F^\bullet))
\end{equation}
and
\begin{equation} \label{eq Ggg commutes with delta}
\Ggg(\delta_{\hat{\imath}}(E^\bullet)) \cong \delta_i(\Ggg(E^\bullet)).
\end{equation}
Thus if $F \in \Coh(X)$ and $E \in \Coh(M)$ are $WIT$ sheaves of indices $n_1$ and $n_2$ respectively, then $\delta_i(F)$ and $\delta_{\hat{\imath}}(E)$ are $WIT$ of indices $n_1$ and $n_2$, respectively, as well.
\end{proposition}

\begin{proof}
We can, for instance, assume that $i$ is antiholomorphic. The proof will be analogous in the holomorphic case. Since $p_X \circ (i \times \hat{\imath}) = i \circ p_X$ and $p_M \circ (i \times \hat{\imath}) = \hat{\imath} \circ p_M$, one has
\begin{align*}
\delta_{\hat{\imath}} ( \Fff(F^\bullet)) & \cong \hat{\imath}^{*}(\overline{\Fff(F^\bullet)} )
\\
& \cong \hat{\imath}^{*} ( \overline{\R (p_M)_{*}(\Uu \stackrel{\L}{\otimes} p_X^{*}F^\bullet}) )
\\
& \cong \R (p_M)_{*}( (i \times \hat{\imath})^{*} \overline{(\Uu \stackrel{\L}{\otimes} p_X^{*}F^\bullet}))
\\
& \cong \R (p_M)_{*}( (i \times \hat{\imath})^{*} \overline{\Uu} \stackrel{\L}{\otimes}  (i \times \hat{\imath})^{*} (\overline{p_X^{*}F^\bullet}))
\\
& \cong \R (p_M)_{*}( (i \times \hat{\imath})^{*} \overline{\Uu} \stackrel{\L}{\otimes}  p_X^{*} i^*(\overline{F^\bullet}))
\\
& \cong \R (p_M)_{*}( (i \times \hat{\imath})^{*} \overline{\Uu} \stackrel{\L}{\otimes}  p_X^{*} \delta_i(F^\bullet)),
\end{align*}
and thanks to Lemma \ref{lm Uu commutes with delta}, 
\[
\R (p_M)_{*}( (i \times \hat{\imath})^{*} \overline{\Uu} \stackrel{\L}{\otimes}  p_X^{*} \delta_i(F^\bullet)) \cong \R (p_M)_{*}( \Uu \stackrel{\L}{\otimes}  p_X^{*} \delta_i(F^\bullet)) \cong \Fff(\delta_i(F^\bullet)),
\]
so \eqref{eq Fff commutes with delta} follows. The proof of \eqref{eq Ggg commutes with delta} is analogous.

Finally, if $\Fff(F)$ and $\Ggg(E)$ are complexes supported on degrees $n_1$ and $n_2$ respectively (i.e. $F$ is WIT of index $n_1$ and $E$ is WIT of index $n_2$), then $\delta_{\hat{\imath}} ( \Fff(F)) \cong \Fff(\delta_i(F))$ and $\delta_i(\Ggg(E^\bullet)) \cong \Ggg(\delta_{\hat{\imath}}(E))$ are supported on degrees $n_1$ and $n_2$ too. Therefore, $\delta_i(F)$ and $\delta_{\hat{\imath}}(E)$ are, respectively, WIT sheaves of indices $n_1$ and $n_2$, respectively.
\end{proof}

\begin{remark}
Given a holomorphic brane involution $i : X \to X$, one can study the behaviour of the structural sheaf $\Oo_{X^i}$ of the fixed point locus under the Fourier--Mukai transform. Suppose that the hypothesis of Proposition \ref{pr FM commutes with brane involutions} are satified, then
\[
\Fff(\Oo_{X^i}) \cong \hat{\imath}^* \Fff(\Oo_{X^i}).
\]
Note that the support of the transformed sheaf $\Fff(\Oo_{X^i})$ is not necessarily contained in the fixed locus of $\hat{\imath}$. Indeed, consider the case of a $\BBB$-involution on a $\K3$ surface where $X^i$ is, by \eqref{eq fixed locus of a BBB-brane}, the union of $8$ isolated points $\{ p_k \}_{k = 1}^8$. The Fourier--Mukai transform of each of the skyscraper sheaves $\Oo_{p_k}$ is $\Fff(\Oo_{p_k}) \cong \Uu|_{\{ p_k \} \times M}$, which is a sheaf supported on the whole $M$.
\end{remark}

Let us now focus on the study of reflexive $\K3$ surfaces. For the rest of the section, suppose that the conditions of Proposition \ref{pr-definition-of-reflexive-K3} hold, so that the moduli space $\hat{X} = \M_X(H,v)$ is a $\K3$ surface, and we have a universal sheaf $\Uu \to X \times \hat{X}$. Taking the brane involution $\hat{\imath} : \hat{X} \to \hat{X}$ one can define another brane involution 
\[
\doublehat{\imath} : X \longrightarrow X, 
\]
by means of \eqref{eq definition of sigma_j holomorphic} and \eqref{eq definition of sigma_j antiholomorphic}. We now prove that this construction is self-dual.

\begin{lemma}
In the previous notation, 
\[
\doublehat{\imath} = i.
\]
\end{lemma}

\begin{proof}
For any $y \in X$, denote by $F_y$ the corresponding sheaf $\Uu^*_y$ on $\hat{X}$. We give the proof for $i$ algebraic antiholomorphic, the holomorhic case is analogous. By the definition of $\doublehat{\imath}$, one has
\[
F_{\doublehat{\imath}(y)} \cong \hat{\imath}^* \overline{F}_y.
\]
Recalling that $i$ is an involution,
\[
\hat{\imath}^* \overline{F}_y \cong \hat{\imath}^* \overline{F}_{i^2(y)}.
\]
Finally, by the definition of $\hat{\imath}$,
\[
\hat{\imath}^* \overline{F}_{i^2(y)} \cong \hat{\imath}^*  \hat{\imath}^* \overline{\overline{F}}_{i(y)} \cong F_{i(y)}.
\]
\end{proof}

To finish the section, we study how the map in integral cohomology induced by a brane involution is transformed under the Hodge isometry of Theorem \ref{tm FM induces a Hodge isometry}.

\begin{proposition}
Consider the Fourier--Mukai partners $X$ and $\hat{X}$, and suppo\-se that $i : X \to X$ is a brane involution. Let us denote by $i^* : \widetilde{H}(X,\ZZ) \to \widetilde{H}(X, \ZZ)$ and $\hat{\imath}^* : \widetilde{H}(\hat{X},\ZZ) \to \widetilde{H}(\hat{X},\ZZ)$ the induced maps in cohomology. Then, the diagram
\[
\xymatrix{
\widetilde{H}(X,\ZZ) \ar[r]^{i^*} \ar[d]_{\psi}^{\cong} & \widetilde{H}(X,\ZZ) \ar[d]^{\psi}_{\cong}
\\
\widetilde{H}(\hat{X},\ZZ) \ar[r]^{\hat{\imath}^*} & \widetilde{H}(\hat{X},\ZZ),
}
\]
commutes, being $\psi$ the Hodge isometry of Theorem \ref{tm FM induces a Hodge isometry}.
\end{proposition}

\begin{proof}
Recall that $p_X \circ (i \times \hat{\imath}) = i \circ p_X$. Then, for each $\alpha \in \widetilde{H}(X,\ZZ)$
\[
\psi(i^*\alpha) = (p_{\hat{X}})_*(p_X^* i^* \alpha )\cdot v(\Uu) = (p_{\hat{X}})_* \left( (i \times \hat{\imath})^*p_X^* \alpha \cdot v(\Uu) \right ).
\]
After Lemma \ref{lm Uu commutes with delta}, one has that $v(\Uu) = (i \times \hat{\imath})^*v(\Uu)$ and therefore
\begin{align*}
\psi(i^*\alpha) = & (p_{\hat{X}})_* \left( (i \times \hat{\imath})^*p_X^* \alpha \cdot v(\Uu) \right )
\\
= & (p_{\hat{X}})_*  (i \times \hat{\imath})^*\left( p_X^* \alpha \cdot v(\Uu) \right ) 
\\
= & \hat{\imath}^* (p_{\hat{X}})_* \left( p_X^* \alpha \cdot v(\Uu) \right ) 
\\
= & \hat{\imath}^* \psi(\alpha),
\end{align*}
since $p_{\hat{X}} \circ (i \times \hat{\imath}) = \hat{\imath} \circ p_{\hat{X}}$.
\end{proof}

\section{Brane involutions and mirror symmetry}
\label{sc mirror symmetry}

\subsection{Mirror symmetry for irreducible holomorphic symplectic manifolds}
\label{definition}
The objective of this section is to study how branes are transformed under mirror symmetry. More precisely, let $X$ be an IHS manifold and $Y$ its mirror. Assuming that $X$ is endowed with a brane involution $i$, can we also endowed $Y$ with a brane involution $\check{\imath}$? Moreover, what will be the type of $\check{\imath}$?

Since the notion of brane is defined with respect to a hyperk\"ahler structure, we need a definition of mirror symmetry which takes into account the hyperk\"ahler structure. That is why, in this section, we work with the notion of mirror symmetry proposed by Huybrechts in \cite[Section 6]{Huybrechts}, which is given in terms of action on the period domain.

Consider a marked IHS manifold $(X,\varphi)$ with $\varphi:H^{2}(X,\Z)\simeq \Gamma$. Denote by $U$ the standard hyperbolic plane $\left (\ZZ^2, \tiny{\begin{pmatrix} 0 & 1 \\ 1 & 0 \end{pmatrix}}\right)$. We assume that there exists an embedding $j: U\hookrightarrow \Gamma$ and consider a basis $(v,v^{*})$ of $j(U)$. Let us denote $M:=(j(U))^{\bot}$ and by $\pr:\Gamma\otimes\RR\rightarrow M\otimes\RR$, the orthogonal projection.

Let $(X, \omega_X, \sigma_X, \varphi)$ be a marked IHS manifold endowed with a hyperk\"ahler structure. We say that $(X,\omega_X,\sigma_X,\varphi)$ is \emph{admissible according to $j$} if
\begin{itemize}
\item[(1)]
$\omega_{X}\in (\varphi^{-1}(M) \otimes\RR) \oplus \RR \varphi^{-1}(v)$; and
\item[(2)]
$\varphi(\Ima \sigma_{X}) \cdot v=0.$
\end{itemize}
If the previous conditions are satisfied, we consider, for a given element $\beta\in (M\otimes \RR) \oplus\RR v$,
\[
\check{\sigma}_{X}:=\pr(\beta+i\varphi(\omega_{X}))-\frac{1}{2}(\beta+i\varphi(\omega_{X}))^{2}v^{*}+v \, \in \,  \Gamma\otimes\C.
\]
Let $\mathcal{M}_{\Gamma}^o$ be the connected component of $\mathcal{M}_{\Gamma}$ containing $(X,\varphi)$ and $\mathscr{P}:\mathcal{M}_{\Gamma}^o\rightarrow \Omega$ the period map.
We have $\C\check{\sigma}_{X} \in \Omega$.
Then, by Theorem \ref{period}, we have 
\[
\check{\Pp}_{X}:=\Ppp^{-1}(\CC \check{\sigma}_{X})\neq \emptyset.
\]
Set also,
\[
\check{\omega}_X := \epsilon\left[\pr(\varphi(\Ima \sigma_X))-(\varphi(\Ima \sigma_X)\cdot \beta)v\right],
\]
where we choose $\epsilon\in \left\{1,-1\right\}$ such that $\psi^{-1}(\check{\omega}_X)$ is in the positive cone of all the elements $(Y,\psi) \in \check{\Pp}_{X}$. This is possible by Section 4 of \cite{Markman}. 

We say that $(X,\sigma_{X},\omega_{X},\varphi)$ admits a \emph{hyperk\"ahler mirror according to $j$ and $\beta$} if there exists $(Y,\psi)\in \check{\Pp}_{X}$ such that $\psi^{-1}(\check{\omega}_X)$ is a K\"ahler class of $Y$. If this is verified, $(Y,\psi^{-1}(\check{\sigma}_X),\psi^{-1}(\check{\omega}_X),\psi)$ is called the \emph{hyperk\"ahler mirror} of $(X,\sigma_{X},\omega_{X},\varphi)$ according to $j$ and $\beta$.

Thanks to item (\ref{it torelli}) of Theorem \ref{Torelli}, the hyperk\"ahler mirror is unique up to isomorphism.


\subsection{Discussion on the existence of the hyperk\"ahler mirror}\label{discussion}

It is possible to establish the existence of hyperk\"ahler mirrors under some circumstances using recent techniques on wall divisors and monodromy developed, in particular, in \cite{Markman0} (see also \cite{Markman}), \cite{NoteKalher} and \cite{Hassett}.

Given an IHS manifold $X$, an automorphism $f:H^*(X,\Z) \stackrel{\cong}{\rightarrow} H^*(X,\Z)$ is said to be a \emph{monodromy operator} if it is a parallel transport operator. The \emph{monodromy group} $\Mon(X) $ is the subgroup of $\GL(H^*(X,\Z))$ consisting of all monodromy operators. We denote by $\Mon^2(X)$ the image of $\Mon(X)$ in the orthogonal group $\Ort(H^2(X,\Z))$ and by $\Mon^2_{Hdg}(X)$ the subgroup of $\Mon^2(X)$ of those automorphisms which preserve the Hodge structure. Finally, let us write $\Ort^+(H^{2}(X,\Z))$ for the subgroup of all orthogonal transformation preserving any given orientation of the three positive directions of $H^{2}(X,\Z)$.

The \emph{positive cone} $\mathcal{C}_X$ of an IHS manifold $X$ is the connected component of the cone of positive classes containing a K\"ahler class. The \emph{K\"ahler cone} denoted by $\mathcal{K}_X$ is the cone containing all the K\"ahler classes. The \emph{birational cone} $\mathcal{B}\mathcal{K}_X$ is the union $\cup f^{-1}\mathcal{K}_{X'}$, where $f$ runs through all birational maps between $X$ and any IHS manifold $X'$.
 
Let $D$ be a divisor on a IHS manifold $X$. We recall from Mongardi \cite{NoteKalher} that $D$ is a \emph{wall divisor} if $D^2<0$ and $f^{-1}\circ g(D^\bot)\cap \mathcal{B} \mathcal{K}_X = \emptyset$, for all marking $f$ and $g$ of $X$ such that $f^{-1} \circ g$ is a parallel transport Hodge isometry. By \cite{NoteKalher}, we can understand the wall divisors in terms of a subset of $\Gamma$.

Let $(X,\varphi)$ be an marked IHS manifold with $\varphi:H^{2}(X,\Z)\simeq \Gamma$.
Let $\mathcal{M}_{\Gamma}^o$ be the connected component of $\mathcal{M}_{\Gamma}$ which contains $(X,\varphi)$. The following is a direct consequence of \cite[Theorem 1.3]{NoteKalher}.

\begin{proposition}\label{wall}
There exists a set $\Delta_X \subset \Gamma$ such that for all $(Y,\psi)\in \mathcal{M}_{\Gamma}^o$, 
$\psi^{-1}(\Delta_X)\cap H^{1,1}(Y,\Z)$ are the wall divisors of $Y$.
\end{proposition}

\begin{remark} \label{rm description of Delta K3}
For a $\K3$ surface $X$, the set $\Delta_X$ is the subset $\Gamma$ consisting of those elements whose square is equal to $-2$ (see \cite[Theorem 5.2, Chapter 8]{lecture-on-K3}). The set $\Delta_X$ is also explicitly described for manifolds of $\K3^{[n]}$-type with $n=2,3,4$ in \cite[Proposition 2.12, Theorems 2.14 and 2.15]{NoteKalher} and for generalized Kummer fourfolds in \cite[Section 1.2]{MongWanTari}.
\end{remark} 

The wall divisors allow us to characterize the K\"ahler cone. The following is straight-forward from \cite[Proposition 1.5]{NoteKalher}

\begin{proposition}\label{Kahler}
Let $X$ be an IHS manifold. Let $\kappa$ be a K\"ahler class of $X$ and $x\in \mathcal{C}_X$. Suppose that $x\cdot D>0$ for all wall divisors $D$ such that $D\cdot \kappa>0$. Then, $x$ is a K\"ahler class.
\end{proposition}

Now let $\Delta_X\subset \Gamma$ be the set described in Proposition \ref{wall}. We define
\[ \Delta_X^{d}:=\left\{\left.D\in \Delta_X\right|\ D^2\ \divides\ 2\divi(D)\right\}, \]
where $\divi(D)$ is the integer $l\geq0$ satisfying $\left\{\left.x\cdot D\right|\ x\in \Gamma \right\}=l\Z$.
We also denote 
\[ \Delta_X^{nd}:=\Delta_X\smallsetminus\Delta_X^{d}. \]
For $u\in  H^{2}(X,\Q)$, $u^2<0$, we define the reflection
$R_u:H^{2}(X,\Q)\rightarrow H^{2}(X,\Q)$ by setting
\[ R_u(x)=x-\frac{2u\cdot x}{u^2}. \]
We make the following additional technical assumption.

\begin{hypothesis}\label{Hypo}
For all $(Y,\psi)\in \mathcal{M}_{\Gamma}^o$ and all $D\in \psi^{-1}(\Delta_X^{d})\cap \Pic (Y)$, one has
$R_D\in\Mon^2(Y)$.
\end{hypothesis}

\begin{remark} \label{rm Hyp 1 is true in some cases}
Hypothesis \ref{Hypo} is verified whenever $\Mon^2(X)=\Ort^+(H^2(X,\Z))$, which holds for $\K3$ surfaces (see for instance \cite[ Proposition 5.5, Chapter 7]{lecture-on-K3}) and for IHS manifolds of $\K3^{[n]}$-type with $n=2$ or $n-1$ equal to a prime power (see \cite[Lemma 4.2]{Markman0}).
\end{remark}


We are finally in position to address whether or not $(X,\omega_X,\sigma_X,\varphi)$ admits a hyperk\"ahler mirror. The idea is to transform $\psi^{-1}(\check{\omega}_X)$ into a K\"ahler class by acting with reflections through wall divisors. It can be seen as a generalization of the proof of the transitivity of the Weyl group on the set of chambers of a $\K3$ surface (see \cite[Corollary 2.9 Chapter 8]{lecture-on-K3}).

\begin{proposition} \label{main}
Let $(X,\omega_X,\sigma_X,\varphi)$ be a marked IHS manifold endowed with a hyperk\"ahler structure which is admissible according to $j$.
Assume that
\begin{itemize}
\item[(1)]
$X$ satisfies Hypothesis \ref{Hypo}, and
\item[(2)]
$\check{\sigma}_{X}\cdot D\neq 0,$ for all $D\in \Delta_X^{nd}$.
\end{itemize}
Then, $(X,\sigma_{X},\omega_{X},\varphi)$ admits a hyperk\"ahler mirror according to $j$ and $\beta$ if and only if 
\[
(\check{\sigma}_{X})^\bot \cap (\check{\omega}_X)^\bot \cap \Delta_X^{d}=\emptyset.
\]
\end{proposition}

\begin{proof}
Let $(Y,\psi)\in \check{\Pp}_{X}$. Let $\kappa$ be a K\"ahler class of $Y$.
Let 
\[
\Delta^{+}:=\left\{\left.D\in \psi^{-1}(\Delta_{X})\cap\Pic(Y)\ \right|\ D\cdot\kappa>0 \right\}.
\]
Since we assume (2), we have
\begin{equation}
\Delta^{+}\subset \psi^{-1}(\Delta_X^{d}).
\label{Delta+}
\end{equation}
If $(\check{\sigma}_{X})^\bot \cap (\check{\omega}_X)^\bot \cap \Delta_X^{d}=\emptyset$
then $\psi^{-1}(\check{\omega}_X)\cdot D\neq0$ for all $D\in \Delta^{+}$.

Using the same technique as \cite[Remark 2.5, Chapter 8]{lecture-on-K3}, we can find $K_{1},...,K_{n}\in \Delta^+$ such that 
$$R_{\left[K_{1}\right]}\circ...\circ R_{\left[K_{n}\right]}(\psi^{-1}(\check{\omega}_X))\cdot D>0,$$
for all $D\in \Delta^+$.
Hence by Proposition \ref{Kahler}, $R_{\left[K_{1}\right]}\circ...\circ R_{\left[K_{n}\right]}(\psi^{-1}(\check{\omega}_X))$ is a K\"ahler class.

Moreover  
by (\ref{Delta+}) and Hypothesis \ref{Hypo}, $R_D\in \Mon^2(Y)$ for all for all $D\in  \Delta^+$.
Furthermore, since $D$ is a divisor, we have $$R_D\in \Mon^2_{Hdg}(Y).$$
Let $R:=R_{\left[K_{1}\right]}\circ...\circ R_{\left[K_{n}\right]}$.
Now, we can consider the marked IHS manifold $(Y, \psi\circ R^{-1})$.
Since all $R_{\left[K_{1}\right]}$ are in $\Mon^2_{Hdg}(Y)$, $R$ is a parallel transport operator and $R^{-1}(\sigma_{Y})=\sigma_{Y}$.
It follows that $(Y, \psi\circ R^{-1})\in \check{\Pp}_X$.

Conversely, assume that $(Y,\psi^{-1}(\check{\sigma}_X), \psi^{-1}(\check{\omega}_X),\psi)$ is the hyperk\"ahler mirror according to $j$ and $\beta$ of $(X,\sigma_{X},\omega_{X},\varphi)$. By definition $\psi^{-1}(\check{\omega}_X)\in \mathcal{K}_Y$. If $(\check{\sigma}_X)^\bot \cap (\check{\omega}_X)^\bot \cap \Delta_X^{d}\neq\emptyset$, then there is $D\in \psi^{-1}(\Delta_X^{d})\cap\Pic Y$ such that $D\cdot \psi^{-1}(\check{\omega}_X)=0$. Hence by definition of wall divisor, $\psi^{-1}(\check{\omega}_X)\notin \mathcal{B}\mathcal{K}_Y$ which is absurd since $\psi^{-1}(\check{\omega}_X)\in \mathcal{K}_Y$. 
\end{proof}

Next, we show that $(X,\omega_X,\sigma_X,\varphi)$ admits a hyperk\"ahler mirror for a large class of elements $\beta\in (M\otimes \RR) \oplus\RR v$ provided that further conditions are met.

\begin{proposition}\label{pr main}
Let $(X,\omega_X,\sigma_X,\varphi)$ be a marked IHS manifold endowed with a hyperk\"ahler structure which is admissible according to $j$. Assume that
\begin{itemize}
\item[(1)]
$X$ satisfies Hypothesis \ref{Hypo},
\item[(2)]
$\Pr(\varphi(\omega_X))\cdot D\neq 0$ or $\Pr(\beta)\cdot D\neq 0$ or $v\cdot D\neq0$ or $v^*\cdot D\neq0$ for all $D\in \Delta_X^{nd}$,
\item[(3)]
$\omega_X^2\in\RR\smallsetminus \Q$,
\item[(4)]
$\Pr (\varphi(\Ree \sigma_X))\cdot D\neq 0 \Rightarrow\Pr(\beta)\cdot D \neq 0$, 
or $\varphi(\sigma_X)\cdot D\neq 0 \Rightarrow \varphi(\Ima \sigma_X)\cdot D\neq 0$, for all $D\in \Delta_X^{d}$.
\end{itemize}
Then, there exists a dense uncountable subset $\Lambda\subset \RR^{*}$ such that $(X,\sigma_{X},\omega_{X},\varphi)$ admits a hyperk\"ahler mirror symmetric according to $j$ and $\lambda\beta$, for all $\lambda\in\Lambda$.
\end{proposition}

\begin{proof}
After Proposition \ref{main}, it is clear that one has to show that our assumptions imply that
\begin{equation} \label{crote1}
\check{\sigma}_X\cdot D\neq 0,
\end{equation}
for all $D \in \Delta_X^{nd}$, and
\begin{equation} \label{crote2}
(\check{\sigma}_X)^\bot \cap (\check{\omega}_X)^\bot \cap \Delta_X^{d}=\emptyset.
\end{equation}

We claim that one can scale $\beta$ by an element $\lambda\in\RR^*$ such that 
\begin{equation} \label{techniquecor1}
(\check{\sigma}_X\cdot D=0) \Rightarrow \left\{\begin{aligned}\Pr(\beta)\cdot D=0\\
\Pr(\varphi(\omega_X))\cdot D=0\\
v\cdot D=0\\
v^*\cdot D=0,
\end{aligned}\right.
\end{equation}
for all $D\in \Delta_X^{nd}$. To see that, consider $\Lambda'$ to be the set of $\lambda\in \RR^{*}$ such that $\lambda^2\beta^2\neq \omega_X^2 $ and for all $D\in\varphi^{-1}(\Delta_X^{nd})$:
\begin{equation}
\left\{\begin{aligned}
& D\cdot\Pr(\lambda\beta)\in \RR\smallsetminus \mathbb{Z}^{*}, \\
& D\cdot\left[\Pr(\lambda\beta)-\frac{1}{2}(\lambda^2\beta^2+\omega_X^2)v\right]\in \RR\smallsetminus \mathbb{Z}^{*},\\
& \frac{2\Pr(\lambda\beta)\cdot D}{\omega_X^2-\lambda^2\beta^2}\in \RR\smallsetminus \mathbb{Z}^{*}.
\end{aligned}
\right.
\label{caca6}
\end{equation}
Since $\omega_X^2$ is irrational, when $D\cdot\Pr(\beta)$ and $\beta^2=0$, $$D\cdot\left[\Pr(\beta)-\frac{1}{2}(\beta^2+\omega_X^2)v\right]\in \RR\smallsetminus \mathbb{Z}^{*}$$
Furthermore, $D\cdot \Pr(\beta)$, $\Pr(\beta)\cdot D-\frac{1}{2}\beta^2v\cdot D$ and $\frac{2\Pr(\beta)\cdot D}{\omega_X^2-\beta^2}$ take only a countable number of values when $D$ varies, hence
$\Lambda'$ is a dense uncountable subset of $\RR^{*}$. We then define the set $\Lambda$ as follows:
\[
\Lambda:=\Lambda', ~~{\rm if} ~~ \beta\cdot \omega_X = 0, ~~{\rm or}
\]
\begin{equation} \label{crote5}
\Lambda:=\left\{\left.\lambda\in \Lambda'\ \right|\ \forall\ D\in\varphi^{-1}(\Delta_X^{nd}),\ \frac{\Pr(D)\cdot\Pr(\varphi(\omega_X))}{\lambda\beta\cdot\omega_X}\in \RR\smallsetminus \mathbb{Z}^{*} \right\} ~~{\rm otherwise.}
\end{equation}
 Since $\frac{\Pr(D)\cdot\Pr(\omega_X)}{\beta\cdot\omega_X}$ takes only a countable number of values when $D$ varies,
$\Lambda$ is a dense uncountable subset of $\RR^{*}$.

We address now the proof of the claim that, after scaling $\beta$ by $\lambda\in \Lambda$, one gets (\ref{techniquecor1}). Then we will have proved (\ref{crote1}) under the hypothesis of the proposition.

Let $D\in \Delta_X^{nd}$ such that $D\cdot \check{\sigma}_X=0$.
Then, $\Ima \check{\sigma}_X\cdot D=0$ implies that
\begin{equation}
\Pr(\varphi(\omega_X))\cdot D+(\varphi(\omega_X)\cdot \lambda\beta)v\cdot D=0.
\label{techniquecor2}
\end{equation}
If $\omega_X\cdot \beta\neq 0$ then by (\ref{crote5}), we have $\Pr(\varphi(\omega_X))\cdot D=0$ and $D\cdot v=0$.
Moreover, we have $\Ree \check{\sigma}_X\cdot D=0$, so
$$\Pr(\lambda\beta)\cdot D+D\cdot v^*=0.$$
Again, by the choice of $\lambda$ made in the first line of (\ref{caca6}), we get $\Pr(\lambda\beta)\cdot D=0$ and $D\cdot v^*=0$.

In the case where $\varphi(\omega_X)\cdot \beta=0$, then $\Ima \check{\sigma}_X\cdot D=0$ only implies $\Pr(\omega_X)\cdot D=0$.
Since we also have $\Ree \check{\sigma}_X\cdot D=0$,
$$\Pr(\lambda\beta)\cdot D-\frac{1}{2}(\lambda^2\beta^2-\omega_X^2)v\cdot D+v^*\cdot D=0.$$
Since $\Pr(\lambda\beta)\cdot D-\frac{1}{2}(\lambda^2\beta^2-\omega_X^2)v\cdot D$ is not integral by the second line of (\ref{caca6}), one has
$$\Pr(\lambda\beta)\cdot D-\frac{1}{2}(\lambda^2\beta^2-\omega_X^2)v\cdot D=0\ \et\ v^*\cdot D=0.$$
Furthermore, by the third line of (\ref{caca6}), $\frac{2\Pr(\lambda\beta)\cdot D}{\omega_X^2-\lambda^2\beta^2}$ is not integral. 
Hence $$\Pr(\lambda\beta)\cdot D=0\ \et\ v\cdot D=0.$$
which completes the proof of the claim in equation \eqref{techniquecor1}; we suppose from now on, without any loss of generality, that it is satisfied.

Now we show that (\ref{crote2}) holds. Take $D\in\Pic (Y)\cap \psi^{-1}(\Delta_X^{d})$ and assume that $\psi^{-1}(\check{\omega}_X)\cdot D=0$. We have $D\in \Pic(Y)$, hence $D\cdot\psi^{-1}(\check{\sigma}_X)=0$. It follows by (\ref{techniquecor1}) that
\begin{equation} \label{c1}
\psi(D)\cdot\Pr(\varphi(\omega_X))=0,\ \psi(D)\cdot v^{*}=0,\ \psi(D)\cdot\Pr(\beta)=0\ \et\ \psi(D)\cdot v=0.
\end{equation}
Moreover $\psi^{-1}(\check{\omega}_X)\cdot D=0$ implies that 
\begin{equation} \label{c2}
\Ima \sigma_X\cdot \varphi^{-1}(\psi(D))=0.
\end{equation}

Assumption (4) on the hypothesis, together with \eqref{c1} and \eqref{c2}, implies that 
$\sigma_X\cdot \varphi^{-1}(\psi(D))=0$. By (\ref{c1}), we also have $\varphi^{-1}(\psi(D))\cdot\omega_X=0$. This is impossible by definition of wall divisors, so (\ref{crote2}) holds, and that completes the proof. 
\end{proof}

If the Picard group $\Pic Y$ of the mirror manifold $Y$ is empty, then, by \cite[Corollary 5.7]{Huybrechts2}, all the classes of the positive cone are K\"ahler classes. In particular, the class $\psi^{-1}(\check{\omega}_X)$ will be K\"ahler, and Hypothesis (1) and condition (4) of Proposition \ref{pr main} become redundant. Moreover, in order to guarantee that $\Pic Y=\emptyset$, we only have to exchange $\Delta_X^{nd}$ for the entire $\Gamma$. We obtain:

\begin{corollary}
Let $(X,\omega_X,\sigma_X,\varphi)$ be a marked IHS manifold endowed with a hyperk\"ahler structure which is admissible according to $j$.
Assume that
\begin{itemize}
\item[(1)]
$\Pr(\varphi(\omega_X))\cdot D\neq 0\ \ou\ \Pr(\beta)\cdot D\neq 0\ \ou\ v\cdot D\neq 0\ \ou\ v^*\cdot D\neq 0$ for all $D\in \Gamma$, and
\item[(2)]
$\omega_X^2\in\RR\smallsetminus \Q$.
\end{itemize}

Then there exists a dense uncountable subset $\Lambda\subset \RR^{*}$ such that $(X,\sigma_{X},\omega_{X},\varphi)$ admits a hyperk\"ahler mirror according to $j$ and $\lambda\beta$, for all $\lambda\in\Lambda$.
\end{corollary}
\begin{remark}
Except for the case when $\RR \Ree \sigma_X \oplus \RR \Ima\sigma_X$ contains an integral class, the condition in (1) on $\omega_X$ is generic. Indeed, the K\"ahler cone $\mathcal{K}_X$ is a full dimensional subset of $H^2(X,\RR)\cap\left\langle\Ree \sigma_X,\Ima \sigma_X\right\rangle^\bot$.
\end{remark}

For $\K3$ surfaces, Hypothesis \ref{Hypo} is true by Remark \ref{rm Hyp 1 is true in some cases} and by Remark \ref{rm description of Delta K3} we have a complete description of $\Delta_X$,
\[
\Delta_X^{d}=\left\{\left.D\in \Pic X\right|\ D^2=-2 \right\}\ \et\ \Delta_X^{nd}=\emptyset.
\]
So, from Proposition \ref{pr main}, we obtain the following consequence.

\begin{corollary} \label{K3case}
Let $(X,\omega_X,\sigma_X,\varphi)$ be a marked $\K3$ surface endowed with a hyperk\"ahler structure which is admissible according to $j$. Assume that 
\begin{itemize}
\item[(1)]
$\left(\Pr (\Ree \sigma_X)\cdot D\neq 0\right) \Rightarrow \left(\varphi^{-1}(\Pr(\beta))\cdot D \neq 0\right)$ or $\sigma_X\cdot D\neq 0 \Rightarrow \Ima \sigma_X\cdot D\neq 0$ for all $D\in H^{2}(X,\Z)$ such that $D^{2}=-2$, and
\item[(2)]
$\omega_X^2\in\RR\smallsetminus \Q$.
\end{itemize}
Then, there exists a dense uncountable subset $\Lambda\subset \RR^{*}$ such that $(X,\sigma_{X},\omega_{X},\varphi)$ admits a hyperk\"ahler mirror according to $j$ and $\lambda\beta$, for all $\lambda\in\Lambda$.
\end{corollary}

\subsection{Transforming brane involutions under mirror symmetry}
\label{sc branes and mirror symmetry}

Let $(X,\sigma_{X},$ $ \omega_{X}, \varphi)$ be a marked IHS manifold endowed with a hyperk\"ahler structure which admits a hyperk\"ahler mirror $(Y,\sigma_{Y},\omega_{Y},\psi)$ according to some $j$ and $\beta$. Let $i$ be a brane involution on $(X,\sigma_X,\omega_X,\varphi)$. There are two different ways to transform $i$ under mirror symmetry; these will be called the \emph{direct mirror transform} and the \emph{indirect mirror transform}.

We define the \emph{direct mirror involution in cohomology} to be the involution on $H^2(Y,\ZZ)$ given by
\begin{equation}
\check{\imath}_{dir}^{\, *}:=\psi^{-1}\circ \varphi \circ i^{*} \circ \varphi^{-1}\circ \psi.
\label{direct}
\end{equation}
Now, assume that $\check{\imath}_{dir}^{\, *}$ can be extended to an involution $\check{\imath}_{dir}$ on $Y$. We call $\check{\imath}_{dir}$ the \emph{direct mirror transform} of $i$.

If $i$ is not a $\BBB$-brane involution, we have another construction analogous to the one considered by Dolgachev \cite{Dolgachev} and Camere \cite{Camere} for holomorphic anti-symplectic involutions. Recall from Remark \ref{rm type in terms T and S} that $T := H^{2}(X,\Z)^{i^{*}}$ and $S:=T^\bot$, where $i^*$ acts trivially on $T$ and inverts $S$. Assume $j:U\hookrightarrow \varphi(S)$, and we denote $M:=j(U)^\bot$. Set
\begin{equation}
\check{\imath}_{ind}^{\, *}|_{\psi^{-1}(j(U)\oplus \varphi(T))}=-\id\ \et\ \check{\imath}_{ind}^{\, *}|_{\psi^{-1}(\varphi(M)\cap \varphi(S))}=\id.
\label{indirect}
\end{equation}
By Corollary 1.5.2 of \cite{Lattice}, $\check{\imath}_{ind}^{\, *}$ extends to an involution on $H^{2}(Y,\Z)$ that we call the \emph{indirect mirror involution in cohomology}. Assume that $\check{\imath}_{ind}^{\,*}$ can be extended to an involution $\check{\imath}_{ind}$ on $Y$. We call $\check{\imath}_{ind}$ the \emph{indirect mirror transform} of $i$. 

\begin{remark}
For a $\BBB$-involution $i$ the invariant lattice $\varphi(T)$ has signature $(3, t-3)$, so the anti-invariant lattice $\varphi(S)$ is negative definite and there is no possible embedding $U \hookrightarrow S$. Therefore, we can not construct $\check{\imath}_{ind}^{\, *}$ from a $\BBB$-involution.
\end{remark}

We want to study when $\check{\imath}_{dir}^{\, *}$ or $\check{\imath}_{ind}^{\, *}$ extend to a brane involution on $Y$. In both cases, we will need another assumption, which we now describe. Let 
\[
\nu: \Aut (X)\rightarrow H^{2}(X,\Z)
\]
the natural morphism; Hassett and Tschinkel have shown in \cite[Theorem 2.1]{Kummer} that $\Ker \nu$ is a deformation invariant. We now make the following hypothesis on the deformation class of $X$.

\begin{hypothesis}\label{injective}
The map $\nu$ is injective.
\end{hypothesis}

\begin{remark} \label{rm Hyp 2 holds in K3 and other IHSM}
This hypothesis holds in the case of $\K3$ surfaces due to the Strong Torelli theorem \cite{Looijenga}, manifolds of $\K3^{[n]}$-type (see \cite[Lemma 1.2]{K3Mong}) and deformations of O'Grady's manifolds of dimension 10 (see \cite{OGMong}).
\end{remark}

In the following lemma we study when an involution in cohomology comes from an involution on the manifold.

\begin{lemma}\label{keylemma}
Let $(X,\sigma_{X},\omega_{X},\varphi)$ be a marked IHS manifold endowed with a hyperk\"ahler structure which satisfies Hypothesis \ref{injective}. Let $\iota \in \Mon^2(X)$ be an involution on $H^2(X,\ZZ)$. Let $V:=\Vect_{\RR}(\omega_{X},\Ree \sigma_{X}, \Ima \sigma_{X})$. Assume that $\iota\otimes\RR$ can be restricted to $V$ and $\iota \otimes\RR|_{V}$ is the identity or a reflection through the rays generated by $\omega_{X}$, $\Ree \sigma_{X}$ or $\Ima \sigma_{X}$.
Then $\iota^*$ can be extended to a brane involution on $X$.
\end{lemma}

\begin{proof}
If $\iota \otimes \RR|_{V}$ is either the identity or the reflection through the ray generated by $\omega_{X}$, the claim is a direct consequence of the Torelli Theorem (see Theorem \ref{Torelli}).

Now let us assume that $\iota \otimes \RR|_{V}$ is the reflection through the ray generated by $\Ree \sigma_{X}$, the proof is analogous for the reflection through $\Ima \sigma_{X}$. Let $g$ and $J_1,J_2,J_3$ be a metric and three complex structures on $X$ such that $\omega_{X}=g(\cdot,J_1(\cdot))$, $\Ree \sigma_{X}=g(\cdot,J_2(\cdot))$ and $\Ima \sigma_{X}=g(\cdot,J_3(\cdot))$. Considering the twistor family, we provide a hyperk\"ahler rotation, $r:X\rightarrow Y$, such that $r^*(J_1)=J_2$, $r^*(J_3)=-J_3$ and $r^*$ is a parallel transport operator. Then $(r^{*})^{-1}\circ \iota \circ r^{*}$ is an involution on $H^{2}(Y,\ZZ)$ which fixes a K\"ahler class of $Y$ and sends the holomorphic 2-form $\sigma_{Y}$ of $Y$ to $- \sigma_{Y}$. Then, by Theorem \ref{Torelli}, $(r^{*})^{-1}\circ \iota \circ r^{*}$ can be extended to a homomorphic involution $i$ on $Y$ such that $i^{*}=(r^{*})^{-1}\circ \iota \circ r^{*}$.
It follows that $i':=r^{-1}\circ i\circ r$ is a brane involution on $X$ such that $(i')^{*}=\iota$.
\end{proof}

We can now study the behaviour of brane involutions under mirror symmetry.

\begin{theorem} \label{tm mirror transform of branes}
Let $(X,\sigma_{X},\omega_{X},\varphi)$ be a marked IHS manifold endowed with a hyperk\"ahler structure satisfying Hypothesis \ref{injective}, and let $i : X \to X$ be a brane involution on it of the type expressed in the first column of the table. Suppose that $(Y,\sigma_{Y},\omega_{Y},\psi)$ is the hyperk\"ahler mirror of $(X,\sigma_{X},\omega_{X},\varphi)$ according to $j$ and $\beta$ satisfying the conditions of the second column of the table hold. 

Then the direct mirror involution $\check{\imath}_{dir}^{\, *}$ in cohomology extends to a brane involution $\check{\imath}$ on the mirror $(Y, \sigma_{Y}, \omega_{Y}, \psi)$ of the type indicated in the third column of the table.

If the indirect mirror involution in cohomology satisfies $\check{\imath}_{ind}^{\, *}\in \Mon^2(Y)$, then it extends to a brane involution $\check{\imath}$ on $(Y,\sigma_{Y},\omega_{Y},\psi)$ of the type indicated in the third column of the table.
$$\begin{tabular}{|c|c|c|c|}
\hline
Type of $i$ & Conditions on $j$ and $\beta$ & Type of $\check{\imath}$ & Construction \\
\hline
$\BBB$ &  $j:U\hookrightarrow \varphi(T)$, $\beta\in \varphi(T)\otimes\RR$ & $\BBB$ & Direct \\
\hline
\multirow{2}{1.2cm}{$\BAA$} &  $j:U\hookrightarrow \varphi(S)$, $\beta\in \varphi(S)\otimes\RR$ & $\AAB$& Direct\\
\cline{2-4}
& $j:U\hookrightarrow \varphi(S)$, $\beta\in \varphi(T)\otimes\RR$ & $\BAA$ & Indirect\\
\hline
$\ABA$ &  $j:U\hookrightarrow \varphi(T)$, $\beta\in \varphi(T)\otimes\RR$ & $\ABA$ & Direct\\
\hline
\multirow{2}{1.2cm}{$\AAB$} & $j:U\hookrightarrow \varphi(S)$, $\beta\in \varphi(S)\otimes\RR$ & $\BAA$ & Direct\\
\cline{2-4}
& $j:U\hookrightarrow \varphi(S)$, $\beta\in \varphi(T)\otimes\RR$ & $\AAB$ & Indirect\\
\hline
\end{tabular}$$
Finally, in the $\ABA$ case it is impossible to extend the indirect mirror involution in cohomology to an involution on the mirror.
\end{theorem}

\begin{proof}
First of all, the computation of the type of $\check{\imath}$ and the conditions on $j$ and $\beta$ follow easily from Remark \ref{rm type in terms T and S} and the definitions of $\check{\omega}_X$ and $\check{\sigma}_X$ given in Section \ref{definition}.

By definition, $(X,\varphi)$ and $(Y,\psi)$ are in the same connected component of the moduli space of marked IHS manifolds. It follows that $\psi^{-1}\circ \varphi$ is a parallel transport operator. Hence the direct mirror involution in cohomology, $\check{\imath}_{dir}^{\, *}$, is in $\Mon^2(Y)$.
In every case considered we have that $\check{\imath}_{dir}^{\, *}$ acts on $V=\Vect_{\RR}(\omega_{Y},\Ree \sigma_{Y}, \Ima \sigma_{Y})$ like either the identity or the reflection through the rays generated by $\omega_{Y}$, $\Ree \sigma_{Y}$ or $\Ima \sigma_{Y}$. Then, the direct case follows from Lemma \ref{keylemma}.

The proof of the indirect case is analogous, but we need to assume that $\check{\imath}_{ind}^{\, *}$ is a monodromy operator since this does not hold automatically. 

The final statement follows from the observation that, in order to extend the involution in cohomology to the whole $Y$, the involution in cohomology has to be contained in $\Ort^+(H^2(Y,\Z))$. Suppose $i$ is an $\ABA$-involution and consider $\check{\imath}_{ind}^{\, *}$. It follows that $j:U\hookrightarrow \varphi(S)$ and to have $\Ree \check{\sigma}_X$ anti-invariant, we need $\beta\in \varphi(T)\otimes\RR$. But then $\Ima \check{\sigma}_X$ and $\check{\omega}_X$ are in $(\varphi(S)\cap M)\otimes\RR$, hence invariant by the action of $\check{\imath}_{ind}^{\,*}$. It follows that the involution in cohomology induced by $\check{\imath}_{ind}^{\,*}$ is not contained in $\Ort^+(H^2(Y,\Z))$ and then it cannot be extended to an involution on $Y$.
\end{proof}

\begin{remark}\label{MonO+}
If $\Mon^2(Y)=\Ort^+(H^2(Y,\Z))$, then $\check{\imath}_{ind}^{\, *}$ is always contained in $\Mon^2(Y)$. This is the case, in particular, when $X$ is a $\K3$ surface (see for instance \cite[ Proposition 5.5, Chapter 7]{lecture-on-K3}) or when $X$ is of $\K3^{[n]}$-type with $n=2$ or $n-1$ a prime power (see \cite[Lemma 4.2]{Markman0}).
\end{remark}

For the sake of completeness, we conclude this section constructing an $\ABA$-involution in the mirror related to the injection $j:U\hookrightarrow \Gamma$. Let $(X,\sigma_{X},\omega_{X},\varphi)$ be a marked IHS manifold endowed with a
hyperk\"ahler structure with $\varphi: H^2(X,\Z)\rightarrow\Gamma$. Assume that $(X,\sigma_{X},\omega_{X},\varphi)$ admits a hyperk\"ahler mirror $(Y,\sigma_{Y},\omega_{Y},\psi)$ according to $j$ and $\beta\in\RR\cdot v$.

Define the involution $\tau^* : H^2(Y,\ZZ) \to H^2(Y,\ZZ)$ setting
\begin{equation} \label{eq sporadic}
\tau^*|_{\psi^{-1}(j(U))}=\id\ \et\ \tau^*|_{\psi^{-1}(M)}=-\id.
\end{equation}
Since we assume that $(Y,\sigma_Y, \omega_Y, \psi)$ is the hyperk\"ahler mirror, we have that $\omega_X \in (\varphi^{-1}(M) \otimes \RR) \oplus \RR \varphi^{-1}(v)$ and $\Ima (\sigma_X) \cdot \varphi^{-1}(v) = 0$. Note that these conditions, together with \eqref{eq sporadic}, imply that $\tau^*$ preserves $\Ree \sigma_Y$ and inverts $\omega_Y$ and $\Ima \sigma_Y$.
If we assume that $Y$ respects Hypothesis \ref{injective} and $\tau\in \Mon^2(Y)$, then, by Lemma \ref{keylemma}, one has that $\tau^*$ extends to a $\ABA$-brane involution on $(Y,\sigma_{Y},\omega_{Y},\psi)$.

\section{Some concrete examples}\label{Examples}
\label{sc examples}

\subsection{Mirror transformation of brane involutions on $\K3$ surfaces}


In the case of $\K3$ surfaces, the existence of a hyperk\"ahler mirror is simplified as stated in Corollary \ref{K3case}. Furthermore, by Remarks \ref{rm Hyp 2 holds in K3 and other IHSM} and \ref{MonO+}, the hypothesis of Theorem \ref{tm mirror transform of branes} hold automatically. Therefore, in order to extend $\check{\imath}_{dir}^{\, *}$ and $\check{\imath}_{ind}^{\, *}$ to involutions on the mirror $\K3$ surface we only have to impose the conditions of the table of Theorem \ref{tm mirror transform of branes}. 


Recall the definitons of $T$ and $S$ as provided in Remark \ref{rm type in terms T and S} and note, thanks to Remark \ref{rm other brane involutions in K3}, that $\varphi(T)$ is a 2-elementary Lorentzian sublattice of $\Gamma = U^3\oplus E_{8}(-1)^{2}$ classified by the integers $(r,a,\delta)$.

\begin{example}\label{ex ABA ABA}
Let $(X,\sigma_{X},\omega_{X},\varphi)$ be a $\K3$ surface with $\omega_{X}^2\in\RR\smallsetminus \Q$ and endowed with a $\ABA$-involution $i$. Suppose that there exists an embedding $j:U\hookrightarrow \varphi(T)$. 

Then, there exists $\beta\in \varphi(T)\otimes\RR$ such that $(X,\sigma_{X},\omega_{X},\varphi)$ admits a hyperk\"ahler mirror accor\-ding to $j$ and $\beta$ endowed with an $\ABA$-involution $\check{\imath}$, which is the direct mirror transform of $i$. 

Both involutions $i$ and $\check{\imath}$ have the same invariants $(r,a,\delta)$, and the admissible values for $(r,a,\delta)$ are represented by the graph in Figure \ref{fig ABA ABA}.


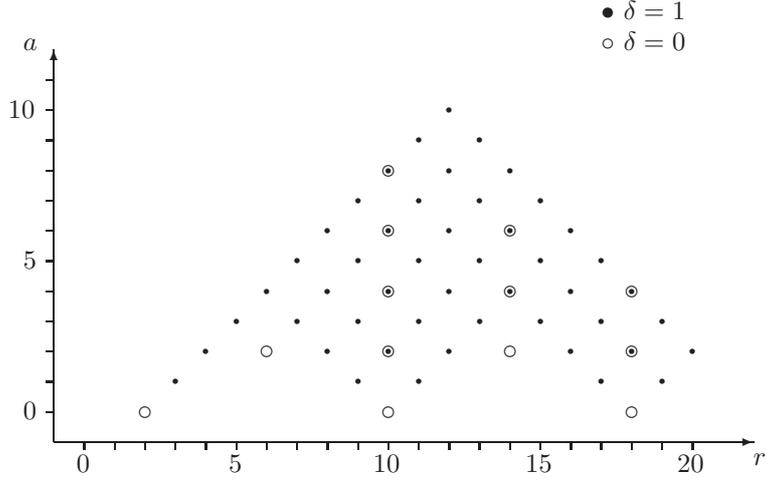
\begin{figure}[H]
\setlength{\unitlength}{2cm}
\centering
$$\begin{picture}(2,3)(1,0)

\put(-0.2,-0.2){\vector(1,0){4.6}}
\put(-0.2,-0.2){\vector(0,1){2.6}}

\put(0,-0.25){\line(0,1){0.05}}
\put(0.2,-0.25){\line(0,1){0.05}}
\put(0.4,-0.25){\line(0,1){0.05}}
\put(0.6,-0.25){\line(0,1){0.05}}
\put(0.8,-0.25){\line(0,1){0.05}}
\put(1,-0.25){\line(0,1){0.05}}
\put(1.2,-0.25){\line(0,1){0.05}}
\put(1.4,-0.25){\line(0,1){0.05}}
\put(1.6,-0.25){\line(0,1){0.05}}
\put(1.8,-0.25){\line(0,1){0.05}}
\put(2,-0.25){\line(0,1){0.05}}
\put(2.2,-0.25){\line(0,1){0.05}}
\put(2.4,-0.25){\line(0,1){0.05}}
\put(2.6,-0.25){\line(0,1){0.05}}
\put(2.8,-0.25){\line(0,1){0.05}}
\put(3,-0.25){\line(0,1){0.05}}
\put(3.2,-0.25){\line(0,1){0.05}}
\put(3.4,-0.25){\line(0,1){0.05}}
\put(3.6,-0.25){\line(0,1){0.05}}
\put(3.8,-0.25){\line(0,1){0.05}}
\put(4,-0.25){\line(0,1){0.05}}

\put(-0.05,-0.4){$0$}
\put(0.95,-0.4){$5$}
\put(1.90,-0.4){$10$}
\put(2.90,-0.4){$15$}
\put(3.90,-0.4){$20$}
\put(4.4,-0.35){$r$}

\put(-0.25,0){\line(1,0){0.05}}
\put(-0.25,0.2){\line(1,0){0.05}}
\put(-0.25,0.4){\line(1,0){0.05}}
\put(-0.25,0.6){\line(1,0){0.05}}
\put(-0.25,0.8){\line(1,0){0.05}}
\put(-0.25,1){\line(1,0){0.05}}
\put(-0.25,1.2){\line(1,0){0.05}}
\put(-0.25,1.4){\line(1,0){0.05}}
\put(-0.25,1.6){\line(1,0){0.05}}
\put(-0.25,1.8){\line(1,0){0.05}}
\put(-0.25,2){\line(1,0){0.05}}
\put(-0.25,2.2){\line(1,0){0.05}}

\put(-0.4,-0.05){$0$}
\put(-0.4,0.95){$5$}
\put(-0.5,1.95){$10$}
\put(-0.4,2.4){$a$}


\put(0.4,0){\circle{.07}}
\put(2,0){\circle{.07}}
\put(3.6,0){\circle{.07}}
\put(2,1.6){\circle{.07}}
\put(2,1.2){\circle{.07}}
\put(2,0.8){\circle{.07}}
\put(2,0.4){\circle{.07}}
\put(1.2,0.4){\circle{.07}}
\put(2.8,0.4){\circle{.07}}
\put(2.8,0.8){\circle{.07}}
\put(2.8,1.2){\circle{.07}}
\put(3.6,0.4){\circle{.07}}
\put(3.6,0.8){\circle{.07}}

\put(0.6,0.2){\circle*{.03}}
\put(1.8,0.2){\circle*{.03}}
\put(2.2,0.2){\circle*{.03}}
\put(3.4,0.2){\circle*{.03}}
\put(3.8,0.2){\circle*{.03}}

\put(0.8,0.4){\circle*{.03}}
\put(1.6,0.4){\circle*{.03}}
\put(2,0.4){\circle*{.03}}
\put(2.4,0.4){\circle*{.03}}
\put(3.2,0.4){\circle*{.03}}
\put(3.6,0.4){\circle*{.03}}
\put(4,0.4){\circle*{.03}}

\put(1,0.6){\circle*{.03}}
\put(1.4,0.6){\circle*{.03}}
\put(1.8,0.6){\circle*{.03}}
\put(2.2,0.6){\circle*{.03}}
\put(2.6,0.6){\circle*{.03}}
\put(3,0.6){\circle*{.03}}
\put(3.4,0.6){\circle*{.03}}
\put(3.8,0.6){\circle*{.03}}

\put(1.2,0.8){\circle*{.03}}
\put(1.6,0.8){\circle*{.03}}
\put(2,0.8){\circle*{.03}}
\put(2.4,0.8){\circle*{.03}}
\put(2.8,0.8){\circle*{.03}}
\put(3.2,0.8){\circle*{.03}}
\put(3.6,0.8){\circle*{.03}}

\put(1.4,1){\circle*{.03}}
\put(1.8,1){\circle*{.03}}
\put(2.2,1){\circle*{.03}}
\put(2.6,1){\circle*{.03}}
\put(3,1){\circle*{.03}}
\put(3.4,1){\circle*{.03}}

\put(1.6,1.2){\circle*{.03}}
\put(2,1.2){\circle*{.03}}
\put(2.4,1.2){\circle*{.03}}
\put(2.8,1.2){\circle*{.03}}
\put(3.2,1.2){\circle*{.03}}

\put(1.8,1.4){\circle*{.03}}
\put(2.2,1.4){\circle*{.03}}
\put(2.6,1.4){\circle*{.03}}
\put(3,1.4){\circle*{.03}}

\put(2,1.6){\circle*{.03}}
\put(2.4,1.6){\circle*{.03}}
\put(2.8,1.6){\circle*{.03}}

\put(2.2,1.8){\circle*{.03}}
\put(2.6,1.8){\circle*{.03}}

\put(2.4,2){\circle*{.03}}


\put(3.4,2.6){$\bullet \ \delta=1$}
\put(3.4,2.4){$\circ \ \delta=0$}
\end{picture}$$\vspace{-0.2cm}

\caption{$\ABA$ to $\ABA$}\label{fig ABA ABA}
\end{figure}

\end{example}

\begin{proof}
Since $i$ is a $\ABA$-involution we have 
\begin{itemize}
\item[(i)]
$\omega_X\in S\otimes\RR$,
\item[(ii)]
$\Ree \sigma_X\in T\otimes\RR$,
\item[(iii)]
$\Ima \sigma_X\in S\otimes\RR$.
\end{itemize}
Hence, from (i) and (iii), $(X,\sigma_{X},\omega_{X},\varphi)$ is admissible according to $j$. Moreover, by (ii), we can find $\beta\in ((\varphi(T)\cap M)\otimes\RR)\oplus\RR v$ such that for all $D\in \varphi(T)$ with $D^2=-2$: $$\left(\Pr (\varphi(\Ree \sigma_X))\cdot D\neq 0\right) \Rightarrow \left(\Pr(\beta)\cdot D \neq 0\right).$$
Then first statement follows from Corollary \ref{K3case} and Theorem \ref{tm mirror transform of branes}.

Since we demand $j:U\hookrightarrow \varphi(T)$, the admissible values $(r,a,\delta)$ are those occurring in the graph of Figure \ref{fig BAA} such that $r\geq a+2$ except $(6,4,0)$. To see this note that, by \cite[Theorem 1.13.5]{Lattice}, the values $(r,a,\delta)$ such that $r\geq a+3$ are admissible. Consider now the values with
$r=a+2$, if $\delta=1$, by \cite[Theorem 3.6.2]{Lattice}, $T$ is isomorphic to the a lattice of the form $U\oplus (-2)^{a}$ and then, it always admit an embedding $U \hookrightarrow T$ in this case. In the rest of the cases,
\begin{itemize}
\item for $(2,0,0)$, we have $T=U$ so, for this value $U \hookrightarrow T$ is admissible;
\item for $(6,4,0)$ we get $T = U(2)\oplus D_4(-1)$ so $U \hookrightarrow T$ is not admissible; and
\item for $(10,8,0)$ we get $T = U\oplus E_8(-2)$, and for this value $U \hookrightarrow T$ is admissible too.
\end{itemize}
\end{proof}


\begin{example}\label{BAAAAB}
Let $(X,\sigma_{X},\omega_{X},\varphi)$ be a $\K3$ surface with $\omega_{X}^2\in\RR\smallsetminus \Q$ and endowed with a $\BAA$-involution $i$. Suppose that there exists an embedding $j:U\hookrightarrow \varphi(S)$. 

Then, there exists $\lambda\in \C^*$ and $\beta\in \varphi(S)\otimes\RR$ such that $(X,\lambda\sigma_{X},\omega_{X},\varphi)$ admits a hyperk\"ahler mirror according to $j$ and $\beta$ endowed with a $\AAB$-involution $\check{\imath}$, which is the direct mirror transform of $i$. 

The involutions $i$ and $\check{\imath}$ have both the same invariants $(r,a,\delta)$ and the admissible values for $(r,a,\delta)$ are represented in Figure \ref{fig BAA AAB}.

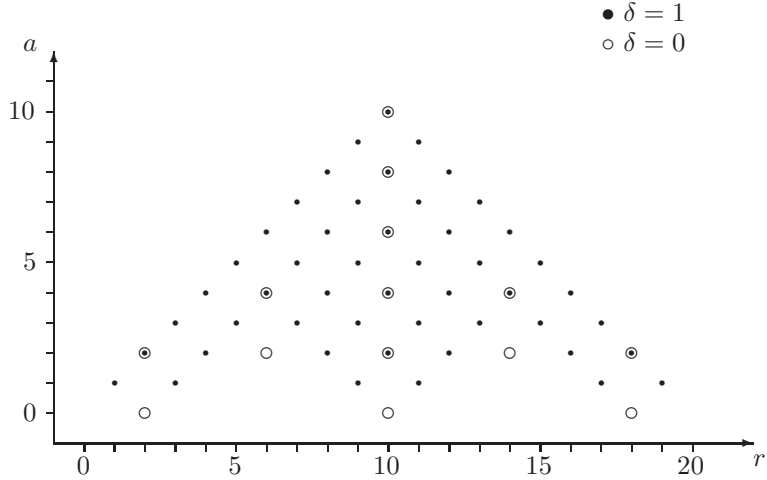
\begin{figure}[H]
\setlength{\unitlength}{2cm}
\centering
$$\begin{picture}(2,3)(1,0)
\put(-0.2,-0.2){\vector(1,0){4.6}}
\put(-0.2,-0.2){\vector(0,1){2.6}}

\put(0,-0.25){\line(0,1){0.05}}
\put(0.2,-0.25){\line(0,1){0.05}}
\put(0.4,-0.25){\line(0,1){0.05}}
\put(0.6,-0.25){\line(0,1){0.05}}
\put(0.8,-0.25){\line(0,1){0.05}}
\put(1,-0.25){\line(0,1){0.05}}
\put(1.2,-0.25){\line(0,1){0.05}}
\put(1.4,-0.25){\line(0,1){0.05}}
\put(1.6,-0.25){\line(0,1){0.05}}
\put(1.8,-0.25){\line(0,1){0.05}}
\put(2,-0.25){\line(0,1){0.05}}
\put(2.2,-0.25){\line(0,1){0.05}}
\put(2.4,-0.25){\line(0,1){0.05}}
\put(2.6,-0.25){\line(0,1){0.05}}
\put(2.8,-0.25){\line(0,1){0.05}}
\put(3,-0.25){\line(0,1){0.05}}
\put(3.2,-0.25){\line(0,1){0.05}}
\put(3.4,-0.25){\line(0,1){0.05}}
\put(3.6,-0.25){\line(0,1){0.05}}
\put(3.8,-0.25){\line(0,1){0.05}}
\put(4,-0.25){\line(0,1){0.05}}

\put(-0.05,-0.4){$0$}
\put(0.95,-0.4){$5$}
\put(1.90,-0.4){$10$}
\put(2.90,-0.4){$15$}
\put(3.90,-0.4){$20$}
\put(4.4,-0.35){$r$}

\put(-0.25,0){\line(1,0){0.05}}
\put(-0.25,0.2){\line(1,0){0.05}}
\put(-0.25,0.4){\line(1,0){0.05}}
\put(-0.25,0.6){\line(1,0){0.05}}
\put(-0.25,0.8){\line(1,0){0.05}}
\put(-0.25,1){\line(1,0){0.05}}
\put(-0.25,1.2){\line(1,0){0.05}}
\put(-0.25,1.4){\line(1,0){0.05}}
\put(-0.25,1.6){\line(1,0){0.05}}
\put(-0.25,1.8){\line(1,0){0.05}}
\put(-0.25,2){\line(1,0){0.05}}
\put(-0.25,2.2){\line(1,0){0.05}}

\put(-0.4,-0.05){$0$}
\put(-0.4,0.95){$5$}
\put(-0.5,1.95){$10$}
\put(-0.4,2.4){$a$}


\put(0.4,0){\circle{.07}}
\put(2,0){\circle{.07}}
\put(3.6,0){\circle{.07}}
\put(2,2){\circle{.07}}
\put(2,1.6){\circle{.07}}
\put(2,1.2){\circle{.07}}
\put(2,0.8){\circle{.07}}
\put(2,0.4){\circle{.07}}
\put(0.4,0.4){\circle{.07}}
\put(1.2,0.4){\circle{.07}}
\put(1.2,0.8){\circle{.07}}
\put(2.8,0.4){\circle{.07}}
\put(2.8,0.8){\circle{.07}}
\put(3.6,0.4){\circle{.07}}

\put(0.2,0.2){\circle*{.03}}
\put(0.6,0.2){\circle*{.03}}
\put(1.8,0.2){\circle*{.03}}
\put(2.2,0.2){\circle*{.03}}
\put(3.4,0.2){\circle*{.03}}
\put(3.8,0.2){\circle*{.03}}

\put(0.4,0.4){\circle*{.03}}
\put(0.8,0.4){\circle*{.03}}
\put(1.6,0.4){\circle*{.03}}
\put(2,0.4){\circle*{.03}}
\put(2.4,0.4){\circle*{.03}}
\put(3.2,0.4){\circle*{.03}}
\put(3.6,0.4){\circle*{.03}}

\put(0.6,0.6){\circle*{.03}}
\put(1,0.6){\circle*{.03}}
\put(1.4,0.6){\circle*{.03}}
\put(1.8,0.6){\circle*{.03}}
\put(2.2,0.6){\circle*{.03}}
\put(2.6,0.6){\circle*{.03}}
\put(3,0.6){\circle*{.03}}
\put(3.4,0.6){\circle*{.03}}

\put(0.8,0.8){\circle*{.03}}
\put(1.2,0.8){\circle*{.03}}
\put(1.6,0.8){\circle*{.03}}
\put(2,0.8){\circle*{.03}}
\put(2.4,0.8){\circle*{.03}}
\put(2.8,0.8){\circle*{.03}}
\put(3.2,0.8){\circle*{.03}}

\put(1,1){\circle*{.03}}
\put(1.4,1){\circle*{.03}}
\put(1.8,1){\circle*{.03}}
\put(2.2,1){\circle*{.03}}
\put(2.6,1){\circle*{.03}}
\put(3,1){\circle*{.03}}

\put(1.2,1.2){\circle*{.03}}
\put(1.6,1.2){\circle*{.03}}
\put(2,1.2){\circle*{.03}}
\put(2.4,1.2){\circle*{.03}}
\put(2.8,1.2){\circle*{.03}}

\put(1.4,1.4){\circle*{.03}}
\put(1.8,1.4){\circle*{.03}}
\put(2.2,1.4){\circle*{.03}}
\put(2.6,1.4){\circle*{.03}}

\put(1.6,1.6){\circle*{.03}}
\put(2,1.6){\circle*{.03}}
\put(2.4,1.6){\circle*{.03}}

\put(1.8,1.8){\circle*{.03}}
\put(2.2,1.8){\circle*{.03}}

\put(2,2){\circle*{.03}}


\put(3.4,2.6){$\bullet \ \delta=1$}
\put(3.4,2.4){$\circ \ \delta=0$}
\end{picture}$$\vspace{-0.2cm}

\caption{$\BAA$ to $\AAB$}\label{fig BAA AAB}
\end{figure}
\end{example}

\begin{proof}
Since $i$ is a $\BAA$-involution we have 
\begin{itemize}
\item[(i)]
$\omega_X\in T\otimes\RR$,
\item[(ii)]
$\Ree \sigma_X\in S\otimes\RR$,
\item[(iii)]
$\Ima \sigma_X\in S\otimes\RR$.
\end{itemize}
If $(\varphi(\Ima\sigma_X))\cdot v=0$, $(X,\sigma_{X},\omega_{X},\varphi)$ is admissible according to $j$ thanks to (i). If $\varphi(\Ima\sigma_X)\cdot v\neq0$, we denote 
\[
\lambda:=\sqrt{-1}-\frac{\varphi(\Ree \sigma_X)\cdot v}{\varphi(\Ima \sigma_X)\cdot v}\in \C^*.
\]
Then $(\Ima\lambda\sigma_X)\cdot v=0$, it follows that $(X,\lambda\sigma_{X},\omega_{X},\varphi)$ is admissible according to $j$.
Moreover by (ii), we can find $\beta\in ((\varphi(S)\cap M)\otimes\RR)\oplus\RR v$ such that for all $D\in \varphi(S)$ with $D^2=-2$ one has 
\[
\left(\Pr (\varphi(\Ree \sigma_X))\cdot D\neq 0\right) \Rightarrow \left(\Pr(\beta)\cdot D \neq 0\right).
\]
We apply Corollary \ref{K3case} and Theorem \ref{tm mirror transform of branes} to prove the first statement.

In order to allow $j:U\hookrightarrow \varphi(S)$, one requires that $(r,a,\delta)$ are the values given in Figure \ref{fig BAA} such that $20-r\geq a$, except $(14,6,0)$. 

We proceed as in Example \ref{ex ABA ABA}. By \cite[Theorem 1.13.5]{Lattice}, the values $(r,a,\delta)$ such that $19-r\geq a$ are admissible. Let consider the values $(r,a,\delta)$ with
$20-r=a$. Equivalently, we have $r(S)=a(S)+2$. If $\delta=1$, by \cite[Theorem 3.6.2]{Lattice}, $S$ is given by a lattice isomorphic to $U\oplus (2)\oplus (-2)^{a-1}$, and, for these values, $U \hookrightarrow S$ is always admissible. Moreover, in the other cases we have
\begin{itemize}
\item  for $(18,2,0)$, we get $S=U\oplus U(2)$, so for this value $U \hookrightarrow S$ is admissible;
\item for $(10,10,0)$, we get $S = U\oplus U(2)\oplus E_{8}(-2)$, then $U \hookrightarrow S$ is admissible as well; and
\item for $(14,6,0)$, we get $S=U(2)^2\oplus D_4$, and, for this value $U \hookrightarrow S$ is not admissible.
\end{itemize}
\end{proof}


\subsection{Mirror transformation of brane involutions on $\K3^{[2]}$-type manifolds}


By Remarks \ref{rm Hyp 2 holds in K3 and other IHSM} and \ref{MonO+}, $\K3^{[2]}$-type manifolds satisfy the hypotheses of Theorem \ref{tm mirror transform of branes} automatically. Therefore, in order to extend $\check{\imath}_{dir}^{\, *}$ and $\check{\imath}_{ind}^{\, *}$ to involutions on the mirror we only have to impose the conditions of the table of Theorem \ref{tm mirror transform of branes}. In addition, the wall divisors are described by Mongardi (see \cite[Proposition 2.12]{NoteKalher}) as
\[
\Delta_X^{d}=\left\{\left.D\in \Pic X\right|\ D^2=-2 \right\} 
\]
and
\begin{equation} \label{eq description Delta fo K@ 2 type nd}
\Delta_X^{nd}=\left\{\left.D\in \Pic X\right|\ D^2=-10\ \et\ \divi(D)=2 \right\}.
\end{equation}

\begin{example}\label{K32ABAABA}
Let $(X,\sigma_{X},\omega_{X},\varphi)$ be an IHS manifold of $\K3^{[2]}$-type with $\omega_{X}^2\in\RR\smallsetminus \Q$ and endowed with a $\ABA$-involution $i$. Suppose that $a(S)=a(T)-1$ and that there exists an embedding $j:U\hookrightarrow \varphi(T)$. 

Then, there exists $\beta\in \varphi(T)\otimes\RR$ such that $(X,\sigma_{X},\omega_{X},\varphi)$ admits a hyperk\"ahler mirror according to $j$ and $\beta$ endowed with a $\ABA$-involution $\check{\imath}$, which is the direct mirror transform of $i$.  

The involutions $i$ and $\check{\imath}$ are both determined by the same integers $(r(T),a(T),\delta(T),$ $ r(S), a(S), \delta(S))$ with admissible values represented by the graph in Figure \ref{K32ABAABA2}.

\begin{figure}[H]
\setlength{\unitlength}{2cm}
\centering
$$\begin{picture}(2,3)(1,0)
\put(-0.4,-0.4){\vector(1,0){4.8}}
\put(-0.4,-0.4){\vector(0,1){2.8}}

\put(-0.2,-0.45){\line(0,1){0.05}}
\put(0,-0.45){\line(0,1){0.05}}
\put(0.2,-0.45){\line(0,1){0.05}}
\put(0.4,-0.45){\line(0,1){0.05}}
\put(0.6,-0.45){\line(0,1){0.05}}
\put(0.8,-0.45){\line(0,1){0.05}}
\put(1,-0.45){\line(0,1){0.05}}
\put(1.2,-0.45){\line(0,1){0.05}}
\put(1.4,-0.45){\line(0,1){0.05}}
\put(1.6,-0.45){\line(0,1){0.05}}
\put(1.8,-0.45){\line(0,1){0.05}}
\put(2,-0.45){\line(0,1){0.05}}
\put(2.2,-0.45){\line(0,1){0.05}}
\put(2.4,-0.45){\line(0,1){0.05}}
\put(2.6,-0.45){\line(0,1){0.05}}
\put(2.8,-0.45){\line(0,1){0.05}}
\put(3,-0.45){\line(0,1){0.05}}
\put(3.2,-0.45){\line(0,1){0.05}}
\put(3.4,-0.45){\line(0,1){0.05}}
\put(3.6,-0.45){\line(0,1){0.05}}
\put(3.8,-0.45){\line(0,1){0.05}}
\put(4,-0.45){\line(0,1){0.05}}

\put(-0.25,-0.6){$0$}
\put(0.75,-0.6){$5$}
\put(1.7,-0.6){$10$}
\put(2.7,-0.6){$15$}
\put(3.7,-0.6){$20$}

\put(4.4,-0.55){$r$}

\put(-0.45,-0.2){\line(1,0){0.05}}
\put(-0.45,0){\line(1,0){0.05}}
\put(-0.45,0.2){\line(1,0){0.05}}
\put(-0.45,0.4){\line(1,0){0.05}}
\put(-0.45,0.6){\line(1,0){0.05}}
\put(-0.45,0.8){\line(1,0){0.05}}
\put(-0.45,1){\line(1,0){0.05}}
\put(-0.45,1.2){\line(1,0){0.05}}
\put(-0.45,1.4){\line(1,0){0.05}}
\put(-0.45,1.6){\line(1,0){0.05}}
\put(-0.45,1.8){\line(1,0){0.05}}
\put(-0.45,2){\line(1,0){0.05}}
\put(-0.45,2.2){\line(1,0){0.05}}

\put(-0.6,-0.25){$0$}
\put(-0.6,0.75){$5$}
\put(-0.7,1.75){$10$}

\put(-0.6,2.4){$a$}


\put(0.4,0){\circle{0.07}}
\put(2,0){\circle{.07}}
\put(3.6,0){\circle{.07}} 
\put(2,1.6){\circle{.07}}
\put(2,1.2){\circle{.07}}
\put(2,0.8){\circle{.07}}
\put(2,0.4){\circle{.07}}
\put(1.2,0.4){\circle{.07}}
\put(1.2,0.8){\circle{.07}}
\put(2.8,0.4){\circle{.07}}
\put(2.8,0.8){\circle{.07}}
\put(2.8,1.2){\circle{.07}}
\put(3.6,0.4){\circle{.07}}
\put(3.6,0.8){\circle{.07}}
\put(0.6,0.2){\circle*{.03}}
\put(1.8,0.2){\circle*{.03}}
\put(2.2,0.2){\circle*{.03}}
\put(3.4,0.2){\circle*{.03}}
\put(3.8,0.2){\circle*{.03}}

\put(0.8,0.4){\circle*{.03}}
\put(1.6,0.4){\circle*{.03}}
\put(2,0.4){\circle*{.03}}
\put(2.4,0.4){\circle*{.03}}
\put(3.2,0.4){\circle*{.03}}
\put(3.6,0.4){\circle*{.03}}
\put(4,0.4){\circle*{.03}}

\put(1,0.6){\circle*{.03}}
\put(1.4,0.6){\circle*{.03}}
\put(1.8,0.6){\circle*{.03}}
\put(2.2,0.6){\circle*{.03}}
\put(2.6,0.6){\circle*{.03}}
\put(3,0.6){\circle*{.03}}
\put(3.4,0.6){\circle*{.03}}
\put(3.8,0.6){\circle*{.03}}

\put(1.2,0.8){\circle*{.03}}
\put(1.6,0.8){\circle*{.03}}
\put(2,0.8){\circle*{.03}}
\put(2.4,0.8){\circle*{.03}}
\put(2.8,0.8){\circle*{.03}}
\put(3.2,0.8){\circle*{.03}}
\put(3.6,0.8){\circle*{.03}}

\put(1.4,1){\circle*{.03}}
\put(1.8,1){\circle*{.03}}
\put(2.2,1){\circle*{.03}}
\put(2.6,1){\circle*{.03}}
\put(3,1){\circle*{.03}}
\put(3.4,1){\circle*{.03}}

\put(1.6,1.2){\circle*{.03}}
\put(2,1.2){\circle*{.03}}
\put(2.4,1.2){\circle*{.03}}
\put(2.8,1.2){\circle*{.03}}
\put(3.2,1.2){\circle*{.03}}

\put(1.8,1.4){\circle*{.03}}
\put(2.2,1.4){\circle*{.03}}
\put(2.6,1.4){\circle*{.03}}
\put(3,1.4){\circle*{.03}}

\put(2,1.6){\circle*{.03}}
\put(2.4,1.6){\circle*{.03}}
\put(2.8,1.6){\circle*{.03}}

\put(2.2,1.8){\circle*{.03}}
\put(2.6,1.8){\circle*{.03}}

\put(2.4,2){\circle*{.03}}


\put(3.4,2.6){$\bullet\ \delta(T)=\delta(S)=1$}
\put(3.4,2.4){$\circ\ \delta(T)=1,\ \delta(S)=0$}
\end{picture}$$\vspace{0.2cm}
\caption{$\ABA$ to $\ABA$, with $a(S)=a(T)-1$}\label{K32ABAABA2}
\end{figure}
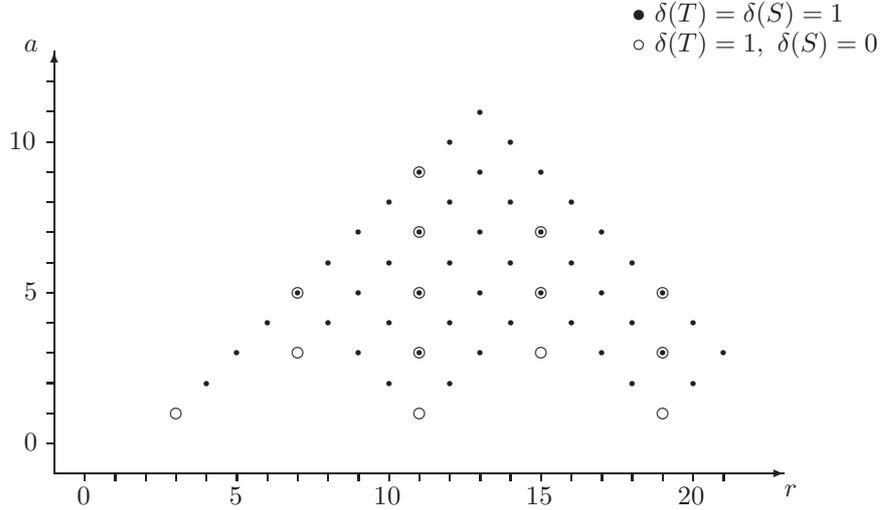

\end{example}

\begin{proof}
Let us recall that the discriminant of a lattice $\discr (L)$ is the absolute value of the determinant of the bilinear form of a lattice $L$. Given a sublattice $S \subset L$ of the same rank, we recall as well the following formula, relating their discriminant with the cardinality of $L/S$
\[
|L/S|^2 = \frac{\discr(L)}{\discr(S)}. 
\]
We have $\discr (\Gamma) = 2$ and by definition of $a$, $\discr (T\oplus S) = 2^{a(T)+a(S)}=2^{1+2a(S)}$. Setting $H_T=\Gamma/T\oplus S$, one has that the cardinality of $H_T$ is $|H_T| = 2^{a(S)}$, so $H_T=(\Z/2\Z)^{a(S)}$. Moreover, since $T$ is a primitive sublattice of $\Gamma$, we have
$H_T\hookrightarrow A_T$. It follows that $A_T=H_T\oplus H_T^\bot$, with $H_T^\bot=\Z/2\Z$.
Necessary, $H_T^\bot\subset A_\Gamma$. Since $A_\Gamma=\Z/2\Z$, one has that
\begin{equation} \label{discrGT}
A_T=H_T\oplus A_\Gamma.
\end{equation}


Recall \eqref{eq description Delta fo K@ 2 type nd} and take a primitive $D\in \Delta_X^{nd}$. Since $\divi(D)=2$, then $\frac{1}{2}D$ is inside $\Gamma^*$ and it provides a non trivial element in $A_\Gamma$ that we denote by $\frac{1}{2}\overline{D}$. By (\ref{discrGT}), we have $A_\Gamma\subset A_T$, so there exists $x\in T$ such that 
\[
\frac{1}{2}\overline{D}=\frac{1}{2}\overline{x}.
\]
Then
\[
\frac{1}{2}D = \frac{1}{2}x + y,
\]
with $y\in \Gamma$, so 
\[
D=x+2y,
\]
The element $x$ is an element non divisible by 2 in $T$ since $\frac{1}{2}\overline{x}$ is not $0$, hence
\begin{equation}
\textnormal{for all } D \in \Delta_X^{nd} \textnormal{ one has that } D \notin S. 
\label{K32equa1}
\end{equation}
  
Since $i$ is an $\ABA$-involution, 
\begin{itemize}
\item[(i)]
$\omega_X\in S\otimes\RR$,
\item[(ii)]
$\Ree \sigma_X\in T\otimes\RR$,
\item[(iii)]
$\Ima \sigma_X\in S\otimes\RR$.
\end{itemize}
From (i) and (iii), $(X,\sigma_{X},\omega_{X},\varphi)$ is admissible according to $j$. Moreover by (ii) and (\ref{K32equa1}), we can find $\beta\in ((\varphi(T)\cap M) \otimes \RR) \oplus \RR v$ such that for all $D\in \varphi(T)$ with $D^2=-2$, one has
\[
\left(\Pr (\varphi(\Ree \sigma_X))\cdot D\neq 0\right) \Rightarrow \left(\Pr(\beta)\cdot D \neq 0\right),
\]
and such that, for all $D\in\Delta_X^{nd}$,
\[
\beta\cdot D\neq 0.
\]
The result follows from Proposition \ref{pr main} and Theorem \ref{tm mirror transform of branes}.

Next, we must find all the values $(r(T),a(T),\delta(T),r(S),a(S),\delta(S))$ which allow an embedding $j:U\hookrightarrow \varphi(T)$. Since we have $a(S)=a(T)-1$, the values considered are the one referenced in Figure \ref{K32BAA2}. In particular, it means that $\delta(T)$ is always 1. As explained in the proof of Example \ref{ex ABA ABA}, by \cite[Theorem 1.13.5]{Lattice}, $r(T)\geq a(T)+3$, we have an embedding $j:U\hookrightarrow \varphi(S)$. In the case $r(T)= a(T)+2$, since $\delta(T)=1$, by \cite[Theorem 3.6.2]{Lattice}, $T$ is isomorphic to $U\oplus (-2)^{a(T)}$. Hence, all values such that $r(T)\geq a(T)+2$ allow an embedding $j:U\hookrightarrow \varphi(T)$.
\end{proof}

\begin{example}\label{K32BAAAAB}
Let $(X,\sigma_{X},\omega_{X},\varphi)$ be an IHS manifold of $\K3^{[2]}$-type with $\omega_{X}^2\in\RR\smallsetminus \Q$ and endowed with a $\BAA$-involution $i$. Suppose that $a(S)=a(T)+1$ and that there exists an embedding $j:U\hookrightarrow \varphi(S)$. 

Then, there exists $\lambda\in \C^*$ and $\beta\in \varphi(S)\otimes\RR$ such that $(X,\lambda\sigma_{X},\omega_{X},\varphi)$ admits a hyperk\"ahler mirror according to $j$ and $\beta$ endowed with a $\AAB$-involution $\check{\imath}$, which is the direct mirror transform of $i$.  

The involutions $i$ and $\check{\imath}$ are both determined by the same integers $(r(T),a(T),\delta(T),$ $ r(S), a(S), \delta(S))$ with admissible values represented by the graph in Figure \ref{K32AAB2}.

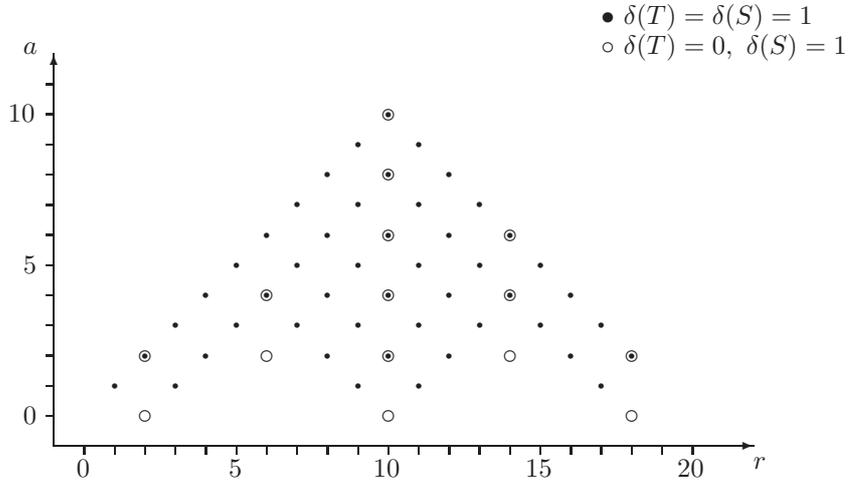
\begin{figure}[H]
\setlength{\unitlength}{2cm}
\centering
$$\begin{picture}(2,3)(1,0)

\put(-0.2,-0.2){\vector(1,0){4.6}}
\put(-0.2,-0.2){\vector(0,1){2.6}}

\put(0,-0.25){\line(0,1){0.05}}
\put(0.2,-0.25){\line(0,1){0.05}}
\put(0.4,-0.25){\line(0,1){0.05}}
\put(0.6,-0.25){\line(0,1){0.05}}
\put(0.8,-0.25){\line(0,1){0.05}}
\put(1,-0.25){\line(0,1){0.05}}
\put(1.2,-0.25){\line(0,1){0.05}}
\put(1.4,-0.25){\line(0,1){0.05}}
\put(1.6,-0.25){\line(0,1){0.05}}
\put(1.8,-0.25){\line(0,1){0.05}}
\put(2,-0.25){\line(0,1){0.05}}
\put(2.2,-0.25){\line(0,1){0.05}}
\put(2.4,-0.25){\line(0,1){0.05}}
\put(2.6,-0.25){\line(0,1){0.05}}
\put(2.8,-0.25){\line(0,1){0.05}}
\put(3,-0.25){\line(0,1){0.05}}
\put(3.2,-0.25){\line(0,1){0.05}}
\put(3.4,-0.25){\line(0,1){0.05}}
\put(3.6,-0.25){\line(0,1){0.05}}
\put(3.8,-0.25){\line(0,1){0.05}}
\put(4,-0.25){\line(0,1){0.05}}

\put(-0.05,-0.4){$0$}
\put(0.95,-0.4){$5$}
\put(1.90,-0.4){$10$}
\put(2.90,-0.4){$15$}
\put(3.90,-0.4){$20$}
\put(4.4,-0.35){$r$}

\put(-0.25,0){\line(1,0){0.05}}
\put(-0.25,0.2){\line(1,0){0.05}}
\put(-0.25,0.4){\line(1,0){0.05}}
\put(-0.25,0.6){\line(1,0){0.05}}
\put(-0.25,0.8){\line(1,0){0.05}}
\put(-0.25,1){\line(1,0){0.05}}
\put(-0.25,1.2){\line(1,0){0.05}}
\put(-0.25,1.4){\line(1,0){0.05}}
\put(-0.25,1.6){\line(1,0){0.05}}
\put(-0.25,1.8){\line(1,0){0.05}}
\put(-0.25,2){\line(1,0){0.05}}
\put(-0.25,2.2){\line(1,0){0.05}}

\put(-0.4,-0.05){$0$}
\put(-0.4,0.95){$5$}
\put(-0.5,1.95){$10$}
\put(-0.4,2.4){$a$}


\put(0.4,0){\circle{.07}}
\put(2,0){\circle{.07}}
\put(3.6,0){\circle{.07}}
\put(2,2){\circle{.07}}
\put(2,1.6){\circle{.07}}
\put(2,1.2){\circle{.07}}
\put(2,0.8){\circle{.07}}
\put(2,0.4){\circle{.07}}
\put(0.4,0.4){\circle{.07}}
\put(1.2,0.4){\circle{.07}}
\put(1.2,0.8){\circle{.07}}
\put(2.8,0.4){\circle{.07}}
\put(2.8,0.8){\circle{.07}}
\put(2.8,1.2){\circle{.07}}
\put(3.6,0.4){\circle{.07}}

\put(0.2,0.2){\circle*{.03}}
\put(0.6,0.2){\circle*{.03}}
\put(1.8,0.2){\circle*{.03}}
\put(2.2,0.2){\circle*{.03}}
\put(3.4,0.2){\circle*{.03}}

\put(0.4,0.4){\circle*{.03}}
\put(0.8,0.4){\circle*{.03}}
\put(1.6,0.4){\circle*{.03}}
\put(2,0.4){\circle*{.03}}
\put(2.4,0.4){\circle*{.03}}
\put(3.2,0.4){\circle*{.03}}
\put(3.6,0.4){\circle*{.03}}

\put(0.6,0.6){\circle*{.03}}
\put(1,0.6){\circle*{.03}}
\put(1.4,0.6){\circle*{.03}}
\put(1.8,0.6){\circle*{.03}}
\put(2.2,0.6){\circle*{.03}}
\put(2.6,0.6){\circle*{.03}}
\put(3,0.6){\circle*{.03}}
\put(3.4,0.6){\circle*{.03}}

\put(0.8,0.8){\circle*{.03}}
\put(1.2,0.8){\circle*{.03}}
\put(1.6,0.8){\circle*{.03}}
\put(2,0.8){\circle*{.03}}
\put(2.4,0.8){\circle*{.03}}
\put(2.8,0.8){\circle*{.03}}
\put(3.2,0.8){\circle*{.03}}

\put(1,1){\circle*{.03}}
\put(1.4,1){\circle*{.03}}
\put(1.8,1){\circle*{.03}}
\put(2.2,1){\circle*{.03}}
\put(2.6,1){\circle*{.03}}
\put(3,1){\circle*{.03}}

\put(1.2,1.2){\circle*{.03}}
\put(1.6,1.2){\circle*{.03}}
\put(2,1.2){\circle*{.03}}
\put(2.4,1.2){\circle*{.03}}
\put(2.8,1.2){\circle*{.03}}

\put(1.4,1.4){\circle*{.03}}
\put(1.8,1.4){\circle*{.03}}
\put(2.2,1.4){\circle*{.03}}
\put(2.6,1.4){\circle*{.03}}

\put(1.6,1.6){\circle*{.03}}
\put(2,1.6){\circle*{.03}}
\put(2.4,1.6){\circle*{.03}}

\put(1.8,1.8){\circle*{.03}}
\put(2.2,1.8){\circle*{.03}}

\put(2,2){\circle*{.03}}


\put(3.4,2.6){$\bullet\ \delta(T)=\delta(S)=1$}
\put(3.4,2.4){$\circ\ \delta(T)=0,\ \delta(S)=1$}
\end{picture}$$\vspace{-0.1cm}

\caption{$\BAA$ to $\AAB$ with $a(S)=a(T)+1$}\label{K32AAB2}
\end{figure}

\end{example}

\begin{proof}
Since $a(S)=a(T)+1$, proceeding as in proof of Example \ref{K32ABAABA} one gets that 
\begin{equation}
\textnormal{for all } D\in\Delta_X^{nd} \textnormal{ one has that } D\notin T. 
\label{K32equa2}
\end{equation}

Now, since $i$ is a $\BAA$-involution, 
\begin{itemize}
\item[(i)]
$\omega_X\in T\otimes\RR$,
\item[(ii)]
$\Ree \sigma_X\in S\otimes\RR$,
\item[(iii)]
$\Ima \sigma_X\in S\otimes\RR$.
\end{itemize}

As in proof of Example \ref{BAAAAB}, if $\varphi(\Ima\sigma_X)\cdot v\neq0$, we multiply $\sigma_{X}$ by $\lambda:=\sqrt{-1}-\frac{\varphi(\Ree \sigma_X)\cdot v}{\varphi(\Ima \sigma_X)\cdot v}$ to have $(X,\lambda\sigma_{X},\omega_{X},\varphi)$ admissible according to $j$.

And by (ii) and (\ref{K32equa2}), we can find $\beta\in \varphi(S)\cap (M\otimes\RR) \oplus \RR v$ such that for all $D\in \varphi(S)$ with $D^2=-2$ one has
\[
\left(\Pr (\varphi(\Ree \sigma_X))\cdot D\neq 0\right) \Rightarrow \left(\Pr(\beta)\cdot D \neq 0\right),
\]
and, for all $D\in\Delta_X^{nd}$, 
\[
\beta\cdot D\neq 0.
\]
The first statement follows from Proposition \ref{pr main} and Theorem \ref{tm mirror transform of branes}.

Finally, we consider the possible values of the parameters
$$ (r(T),a(T),\delta(T),r(S),a(S),\delta(S)). $$ 
For the same reason as in the proof of Example \ref{K32ABAABA}, to allow an embedding $j:U\hookrightarrow \varphi(S)$, we need $r(S)\geq a(S)+2$. That is $20-r(T)\geq a(T)$.
\end{proof}

\subsection{Natural $\ABA$ and $\AAB$-branes on $\K3^{[2]}$}

We now study the fixed locus of $\hat{\imath} : \Hilb^2(X) \to \Hilb^2(X)$ in the case where $i$ is an antiholomorphic (anti)sym\-ple\-ctic involution of $X$. Recall from Remark \ref{rm other brane involutions in K3} and Theorem \ref{tm nikulin description} that the fixed locus is either empty or the disjoint union of smooth surfaces
\begin{equation} \label{eq fixed locus of i}
X^i = \bigsqcup_{k = 1}^n \Sigma_k.
\end{equation}

Denote the Hilbert--Chow morphism by $\varepsilon : \Hilb^2(X) \to \Sym^2(X)$. In this case, it coincides with the blow-up of $\Sym^2(X)$ along the diagonal $\Delta \subset \Sym^2(X)$. Note that the fixed locus of $i_{(2)} : \Sym^2(X) \to \Sym^2(X)$ is 
\begin{equation} \label{eq fixed locus of i_n}
\Sym^2(X)^{i_{(2)}} = \left(  \,  (X/i) \, \cup  \bigsqcup_{k = 0}^n \Sym^2(\Sigma_k) \right ) \sqcup  \bigsqcup_{1 \leq k < k' \leq n}
\Sym(\Sigma_k,\Sigma_{k'}),
\end{equation}
where $\Sym(\Sigma_{k}, \Sigma_{k'})$ is the projection of $\Sigma_{k} \times \Sigma_{k'}$ inside $\Sym^2(X)$. Denote $\Delta_k := \Delta \cap \Sym^2(\Sigma_{k}) \cong \Sigma_{k}$ and observe that $X/i$ intersects with $\Sym^2(\Sigma_{k})$ at $\Delta_k \cong \Sigma_{k}$. In fact, the intersection of the fixed locus $(\Sym^2(X)^{i_{(2)}}$ with the diagonal $\Delta$ is precisely $\bigsqcup_{k = 1}^n \Delta_k \cong X^i$.

\begin{proposition} \label{pr fixed locus of AAB}
Let $X$ be a smooth connected quasi-projective complex surface and $j$ an antiholomorphic involution on it with non-empty fixed locus \eqref{eq fixed locus of i}. Then, the fixed locus of $\hat{\imath} : \Hilb^2(X) \to \Hilb^2(X)$ is
\[
\Hilb^2(X)^{\hat{\imath}} \cong \, \Blow^\RR_{\bigsqcup \Delta_k} \left ( (X/i) \cup \bigsqcup_{k = 0}^n  \Sym^2(\Sigma_{k}) \right ) \sqcup  \bigsqcup_{1 \leq k < k' \leq n}
\Sigma_{k} \times \Sigma_{k'},
\]
where the first component is the (real) blow-up of the singular differentiable manifold $ X/i \cup \bigsqcup_{k = 0}^n  \Sym^2(\Sigma_{k})$ of real dimension 4 along the submanifold $\bigsqcup_{k = 1}^n \Delta_k$, and it is smooth.
\end{proposition}

\begin{proof}
Recall the fixed locus of $i_{(2)}$ described in \eqref{eq fixed locus of i_n}. Since the components $\Sym(\Sigma_{k}, \Sigma_{k'}) \cong \Sigma_{k} \times \Sigma_{k'}$ do not intersect $\Delta$, one has for each of them that the pre-image under the blow-up morphism $\varepsilon$ given by the Hilbert--Chow map is
\[
\varepsilon^{-1}(\Sym(\Sigma_{k}, \Sigma_{k'})) \cong \Sym(\Sigma_{k}, \Sigma_{k'}) \cong \Sigma_{k} \times \Sigma_{k'}.
\]

Take a point $x \in X^i$, the differential $di_x : T_xX \to T_xX$ is an antiholomorphic involution and therefore
\[
\overline{\imath}_x : \PP(T_xX) \longrightarrow \PP(T_xX) 
\]
is antiholomorphic as well. Recall that $\PP(T_xX) \cong \PP^1$ and note that $\overline{\imath}_x$ fixes a circle given by the projection onto $\PP(T_xX)$ of the subspace $(T_xX)^{di_x}$. This circle coincides with the real projective space of $(T_xX)^{di_x}$
\[
\PP(T_xX)^{\overline{\imath}_x} \cong S^1 \cong \PP_\RR \left ( (T_xX^{di_x} ) \right ).
\]
Then, for every $x \in X^i$ (so $[x,x]_{\sym_2} \in \bigsqcup_{k = 1}^n \Delta_k$) one has that the restriction of the blow-up over a point to the fixed locus is
\begin{equation} \label{eq description of the fixed locus of the fibre of HM}
\varepsilon^{-1}([x,x]_{\sym^2}) \cap \Hilb^2(X)^{\hat{\imath}} = \PP_\RR \left ( (T_xX^{di_x} ) \right ) = \PP_\RR \left ( T_x(X^{i}) \right ).
\end{equation}
Therefore the restriction of the blow-up of $\Sym^2(X)$ at $\Delta$ to the fixed locus of $\hat{\imath}$ is the real blow-up at $\bigsqcup_{k = 1}^n \Delta_k \cong X^i$,
\begin{align*}
\Blow_{\Delta}\left ( \Sym^2(X) \right) \cap \varepsilon^{-1} & \left ( (X/i) \cup \bigsqcup_{k = 0}^n  \Sym^2(\Sigma_{k}) \right ) =
\\
= & \Blow^\RR_{\bigsqcup \Delta_i} \left ( X/i \cup \bigsqcup_{k = 0}^n  \Sym^2(\Sigma_{k}) \right ).
\end{align*}
Finally, recall from \cite{fogarty} that $\Hilb^n(X)$ is smooth. Then, thanks to \cite[Lemma 9]{baraglia&schaposnik}, this component is also smooth.
\end{proof}

This, together with Remark \ref{rm other brane involutions in K3} and Theorem \ref{tm nikulin description}, gives an explicit description of the fixed locus of the natural $\ABA$ and $\AAB$-involutions inside $\Hilb^2(X)$.

\begin{corollary}
Let $X$ be a smooth $\K3$-surface and $i$ an antiho\-lo\-morphic (anti)sym\-plec\-tic involution on it associated to the lattice invariants $(r,a,\delta)$. Then, 
\begin{enumerate}
\item for $r = 10$, $a = 10$ and $\delta = 0$, one has
\[
\Hilb^2(X)^{\hat{\imath}} \cong \,  X/i, 
\]
which is a smooth $4$-manifold with no boundary,

\item for $r = 10$, $a = 8$ and $\delta = 0$, one has
\[
\Hilb^2(X)^{\hat{\imath}} \cong \, M \,  \sqcup (T_{1} \times T_{2}),
\]
where $T_1$ and $T_2$ are smooth $2$-dimensional tori, and $M$ is a smooth $4$-manifold with no boundary, and

\item in the remaining cases, setting $g = (22 - r - a)/2$ and $\ell = (r - a)/2$, one has
\[
\Hilb^2(X)^{\hat{\imath}} \cong \, M'   \, \sqcup \bigsqcup_{k = 1}^\ell
(\Sigma_g \times S^2)  \sqcup \bigsqcup_{k = 1}^{\binom{\ell}{2}} (S^2 \times S^2),
\]
where $S^2$ is the $2$-sphere, $\Sigma_g$ is a compact smooth surface of genus $g$ and $M'$ is a smooth $4$-manifold with no boundary.
\end{enumerate}
\end{corollary}

\

\noindent
Emilio \textsc{Franco}

\noindent
CMUP, Universidade do Porto

\noindent
Rua do Campo Alegre 1021/1055, Porto (Portugal)

\noindent
{\tt emilio.franco@fc.up.pt}

\

\noindent
Marcos \textsc{Jardim}

\noindent
IMECC, Universidade Estadual de Campinas

\noindent
Rua S\'ergio Buarque de Holanda 651, Cidade Universit\'aria "Zeferino Vaz", Campinas (SP, Brazil)

\noindent
{\tt jardim@ime.unicamp.br}

\

\noindent
Gr\'egoire \textsc{Menet}

\noindent
IMECC, Universidade Estadual de Campinas

\noindent
Rua S\'ergio Buarque de Holanda 651, Cidade Universit\'aria "Zeferino Vaz", Campinas (SP, Brazil)

\noindent
{\tt menet@ime.unicamp.br}

\end{document}